\documentclass[11pt]{amsart}

\usepackage{amsmath}
\usepackage{amscd}
\usepackage{amsthm}
\usepackage{graphicx}
\usepackage{amssymb}
\usepackage{epstopdf}
\theoremstyle{plain}
\newtheorem{Thm}{Theorem}[section]
\newtheorem{Conj}[Thm]{Conjecture}
\newtheorem{Prop}[Thm]{Proposition}
\newtheorem{Cor}[Thm]{Corollary}
\newtheorem{Lem}[Thm]{Lemma}

\theoremstyle{definition}

\newtheorem{Rem}[Thm]{Remark}
 
\numberwithin{equation}{section}

\title{Non-commutative deformations of perverse coherent sheaves and rational curves}
\author{Yujiro Kawamata}
\usepackage{fancyhdr}
\begin{document}
\maketitle

\tableofcontents

%\dynkin{D}{4}

%\xygraph{
    %\bullet ([]!{+(0,-.3)} {\alpha_1}) - [r]
    %\bullet ([]!{+(0,-.3)} {\alpha_2}) - [r] \cdots - [r]
    %\bullet ([]!{+(0,-.3)} {\alpha_{n - 1}}) - [r]
    %\bullet ([]!{+(0,-.3)} {\alpha_n})}

\begin{abstract}
We consider non-commutative deformations of sheaves on algebraic varieties.
We develop some tools to determine parameter algebras of versal non-commutative deformations for
partial simple collections and the structure sheaves of smooth rational curves.
We apply them to universal flopping contractions of length $2$ and higher. 
We confirm Donovan-Wemyss conjecture in the case of deformations of Laufer's flops.
\end{abstract}

14A22, 14F05, 14J30, 16S38.

%%%%%%%%%%%%%%%%%%%%%%
%%%%%%%%%%%%%%%%%%%%%%
%%%%%%%%%%%%%%%%%%%%%%
\section{Introduction}

The purpose of this paper to develop tools for calculating 
versal non-commutative (NC) deformations.
It is a continuation of our papers \cite{NC multi}, \cite{NC perv} and \cite{NC formal}.
We develop two methods; one to determine the versal NC deformation as well as the NC deformation algebra 
for a partial simple collection of perverse coherent sheaves, and another to determine the degree $2$ parts of 
defining equations of the NC deformation algebra for a smooth rational curve on a smooth variety.  
Then we apply these methods to examples of flopping contractions of rational curves in the case of length $2$ and higher.

\vskip 1pc 

We assume that the base field $k$ is an algebraically closed field of characteristic $0$.
The first method is a generalization of Theorems 6.1 and 6.2 of \cite{NC perv}.
We consider a projective morphism $f: Y \to X$ such that $X = \text{Spec}(R)$ 
for a complete local Noetherian $k$-algebra $R$ whose residue field is $k$.
We assume that there is a locally free coherent sheaf $P$ on $Y$ which is a tilting generator of $D^b(\text{coh}(Y))$
as in \cite{NC perv}.
We define a category of perverse coherent sheaves $\text{Perv}(Y/X)$ to be the one corresponding to 
the category of modules $(\text{mod-}A)$ under Bondal-Rickard derived equivalence $D^b(\text{coh}(Y)) \cong D^b(\text{mod-}A)$.
where $A = f_*\mathcal{E}nd(P)$ is a coherent sheaf of associative $\mathcal{O}_X$-algebras.
Let $\{s_j\}_{j=1}^m$ be the set of all simple objects in $\text{Perv}(Y/X)$, and take arbitrary
non-empty subset $J \subset \{1,\dots,m\}$.
Then we determine explicitly the versal multi-pointed NC deformation of the simple collection $\{s_j\}_{j \in J}$ as well as
the NC deformation algebra, the parameter associative algebra of the versal deformation
(Theorem~\ref{partial}).
(\cite{NC perv}~Theorem 6.1 treated the case when $J = \{1,\dots,m\}$, and Theorem 6.2 when $f$ has at most $1$-dimensional fibers 
and $J$ is the complement of an element which corresponds to the structure sheaf). 
This is applied in \S 4 for determining NC deformations of a non-reduced fiber of a flopping contraction.

\vskip 1pc 

The second method concerns formal NC deformations of a smooth rational curve $C$ on a smooth variety $X$.
In this case, the NC deformation algebra is a quotient ring of a NC power series ring 
$k\langle \langle \text{Ext}^1(\mathcal{O}_C,\mathcal{O}_C)^* \rangle \rangle$ divided by an ideal $I$ coming from an obstruction space
$\text{Ext}^2(\mathcal{O}_C,\mathcal{O}_C)$ (cf. \cite{NC formal}).
An example by Donovan-Wemyss (\cite{Donovan-Wemyss}~Example 1.3) 
shows that a certain flopping curve on a smooth 3-fold has a deformation algebra 
$k\langle \langle a,b \rangle \rangle/(ab+ba, a^2-b^3)$.
The composition map
\[
\text{Ext}^1(\mathcal{O}_C,\mathcal{O}_C) \otimes \text{Ext}^1(\mathcal{O}_C,\mathcal{O}_C) 
\to \text{Ext}^2(\mathcal{O}_C,\mathcal{O}_C)
\]
determines the second order terms of the generators of $I$.
We calculate the behavior of this map in terms of the positivity and the negativity of the normal bundle $N_{C/X}$
(Theorem~\ref{symmetric} and Proposition~\ref{skew-symmetric}).
The reason why anti-symmetric relations like $ab+ba=0$ appear in the NC deformation algebra 
is revealed to be the mixture of the positivity and the negativity of the normal bundle.

\vskip 1pc 

We apply these two results to investigate NC deformations of fibers for some flopping contractions.
We consider projective birational morphisms $f: Y \to X$ from smooth varieties with at most $1$-dimensional fibers  
and such that the canonical divisors $K_Y$ are relatively trivial.

The first case is a universal flopping contraction of length $2$ constructed by Curto-Morrison (\cite{Curto-Morrison}).
We have $\dim X = 6$, and the scheme theoretic central fiber $f^{-1}(0)$ has multiplicity $2$ in this case. 
We determine NC deformation rings of the reduced central fiber and that of the scheme theoretic non-reduced fiber
(Theorems~\ref{univ def alg length 2} and \ref{univ def alg length 2-2}).
They correspond to the two irreducible components of the singular locus $\text{Sing}(X)$ of the base space $X$.
When we consider only commutative deformations, then there are no obstructions in both cases.
But their NC deformations look very different; one has more NC deformations over an irreducible component of $\text{Sing}(X)$ 
which is again singular, 
and the other has only commutative deformations over another irreducible component of $\text{Sing}(X)$ which is non-singular.  

\vskip 1pc 

Next we consider deformations of Laufer's flopping contraction. 
We consider a family of flopping contractions of $3$-dimensional varieties over $2n$-dimensional affine space 
for a positive integer $n$.
This family of flops of length 2 was also considered independently by Van Garderen \cite{vGarderen} in
the context of Donaldson-Thomas invariants.
We calculate the NC deformation algebra (Theorem~\ref{Laufer def}) using Theorem~\ref{univ def alg length 2}.
We prove that the base affine space has a stratification such that 
there are only $2n+1$ isomorphism types in this deformation family (Theorem~\ref{stratification}), and then prove
that the isomorphism types of the NC deformation algebras correspond bijectively to those of flopping contractions
(Proposition~\ref{non-isomorphism1} and Theorem~\ref{non-isomorphism2}).
This is an affirmative answer to a conjecture of Donovan-Wemyss (\cite{Donovan-Wemyss}~Conjecture 1.4) in this case.

\vskip 1pc 

Finally we determine NC deformation algebras for higher length universal flopping contractions (Theorem~\ref{higher length}) 
using a classification result of Karmazyn (\cite{Karmazyn}).
There has been a problem to understand all flopping rational curves on smooth $3$-folds since 1980's.
It looks a simple problem, but is actually quite complicated probably because of its non-commutative nature, 
which is indicated by the difference of the proofs of Proposition~\ref{non-isomorphism1} and Theorem~\ref{non-isomorphism2}.

\vskip 1pc 

The author would like to thank the referee for the careful reading. 
This work was partly supported by JSPS Grant-in-Aid 16H02141.

%%%%%%%%%%%%%%%%%%%%%%
%%%%%%%%%%%%%%%%%%%%%%
%%%%%%%%%%%%%%%%%%%%%%
\section{Versal deformations of partial simple collections}

Let $k$ be an algebraically closed field of characteristic $0$, and
let $f: Y \to X$ be a projective morphism of Noetherian $k$-schemes such that $X = \text{Spec}(R)$ for a complete local algebra $R$
whose residue field is $k$.
A locally free coherent sheaf $P$ on $Y$ is called a {\em tilting generator} if the following conditions are satisfied: 
(1) all higher direct images of $\mathcal{E}nd(P)$ for $f$ vanishes, (2) $P$ generates the 
derived category of quasi-coherent sheaves $D(\text{Qcoh}(Y))$. 
Then the derived Morita equivalence theorem of Bondal and Rickard (\cite{Bondal}, \cite{Rickard}) 
tells us that there is an equivalence of triangulated categories
\[
\Phi: D^b(\text{coh}(Y)) \cong D^b(\text{mod-}A)
\]
given by $\Phi(\bullet) = R\text{Hom}(P,\bullet)$, 
where $A = f_*\mathcal{E}nd(P)$ is a coherent sheaf of associative 
$\mathcal{O}_X$-algebras.
The category of {\em perverse coherent sheaves} $\text{Perv}(Y/X)$ is defined to be the abelian subcategory of 
$D^b(\text{coh}(Y))$ corresponding to the category of finitely generated right $A$-modules $(\text{mod-}A)$ under this equivalence 
(\cite{NC perv}).
We note that $P$ becomes a projective object in $\text{Perv}(Y/X)$, because $\Phi(P) \cong A$.
For example, the category of perverse coherent sheaves ${}^p\text{Perv}(Y/X)$ by Bridgeland \cite{Bridgeland} and Van den Bergh \cite{VdBergh}
is a special case where $f$ has at most $1$-dimensional fibers.

The following is a generalization of \cite{Donovan-Wemyss}~Definition 2.9, Definition 3.8, Lemma 3.9, 
and \cite{NC perv}~Theorems 6.1 and 6.2.
(For a flopping contraction of a smooth $3$-fold, the contraction algebra and the NC deformation algebra are defined 
in Definitions 2.9 and 3.8 of \cite{Donovan-Wemyss} as the factor algebra and the endomorphism algebra as in the theorem below, 
and they are shown to coincide in Lemma 3.9 of [6].  
\cite{NC perv}~Theorem 6.1 treated the case when $J = \{1,\dots,m\}$, and Theorem 6.2 when $f$ has at most $1$-dimensional fibers  
and $J$ is the complement of an element which corresponds to the structure sheaf). 

\begin{Thm}\label{partial}
Let $\{s_j\}_{j=1}^m$ be the set of all simple objects in $\text{Perv}(Y/X)$ above the closed point $x_0 \in X$, and
let $P = \bigoplus_{i=1}^m P_i$ be the direct sum of all 
indecomposable projective objects in the category of perverse coherent sheaves $\text{Perv}(Y/X)$ 
such that $\dim \text{Hom}(P_i,s_j) = \delta_{ij}$.
Let $J \subset \{1,\dots,m\}$ be any non-empty subset, and let $J^c = \{1,\dots,m\} \setminus J$ be the complement.
Denote $P_J = \bigoplus_{i \in J} P_i$ and $P_{J^c} = \bigoplus_{i \in J^c} P_i$.
Let $A = \text{End}(P)$ be an associative algebra of endomorphisms, 
and let $I \subset A$ be the two-sided ideal generated by $g \in A$ which factorizes in the form
$P \to P_{J^c} \to P$ as $\mathcal{O}_Y$-homomorphisms.

(1) Let $Q$ be defined by the following distinguished triangle in $D^b(\text{coh}(Y))$:
\[
Q[-1] \to \text{Hom}(P_{J^c},P) \otimes_{\text{End}(P_{J^c})} P_{J^c} \to P \to Q
\]
where the tensor product is defined in $\text{Perv}(Y/X)$ (see the proof) and the morphism is a natural one.
Then $Q \in \text{Perv}(Y/X)$.

(2) $Q$ is a versal NC deformation of a simple collection $\bigoplus_{j \in J} s_j$.

(3) The parameter algebra $A_{\text{def}} := \text{End}(Q)$ of the versal NC deformation is given by $A/I$.
\end{Thm}

\begin{proof}
(1) We denote $B = \text{End}(P_{J^c})$, and let 
\[
F_1 \to F_0 \to \text{Hom}(P_{J^c},P) \to 0
\]
be a resolution by free right $B$-modules.
Then we define a tensor product $\text{Hom}(P_{J^c},P) \otimes_B P_{J^c}$ in $\text{Perv}(Y/X)$
by an exact sequence in $\text{Perv}(Y/X)$:
\[
F_1 \otimes_B P_{J^c} \to F_0 \otimes_B P_{J^c} \to \text{Hom}(P_{J^c},P) \otimes_B P_{J^c} \to 0.
\]

In order to prove that $Q \in \text{Perv}(Y/X)$, it is sufficient to show that
the morphism $\text{Hom}(P_{J^c},P) \otimes_B  P_{J^c} \to P$ is injective in $\text{Perv}(Y/X)$.
It is in turn sufficient to prove that the homomorphism which is obtained by applying $\text{Hom}(P,\bullet)$:
\[
\text{Hom}(P_{J^c},P) \otimes_B \text{Hom}(P,P_{J^c}) \to \text{Hom}(P,P) 
\]
is injective, because $P$ is a projective generator of $\text{Perv}(Y/X)$.
We have an injective homomorphism of a direct summand
\[
\text{Hom}(P_{J^c},P) \otimes_k \text{Hom}(P,P_{J^c}) \to \text{Hom}(P,P) \otimes_k \text{Hom}(P,P).
\]
The right (resp. left) action of $A$ on $\text{Hom}(P,P)$ induces the right (resp. left) action of $B$ on a direct summand 
$\text{Hom}(P_{J^c},P)$ (resp. $\text{Hom}(P,P_{J^c})$).
Therefore the homomorphism
\[
\text{Hom}(P_{J^c},P) \otimes_B \text{Hom}(P,P_{J^c}) \to \text{Hom}(P,P) \otimes_A \text{Hom}(P,P)
\]
is injective, hence our claim, because $\text{Hom}(P,P) \otimes_A \text{Hom}(P,P) \cong \text{Hom}(P,P)$.

\vskip 1pc

(2) Since $\text{Hom}(P_{J^c}, \bullet)$ is an exact functor on $\text{Perv}(Y/X)$, we have an exact sequence
\[
F_1 \to F_0 \to \text{Hom}(P_{J^c}, \text{Hom}(P_{J^c},P) \otimes_B P_{J^c}) \to 0.
\]
Hence 
\[
\text{Hom}(P_{J^c},P) \cong \text{Hom}(P_{J^c}, \text{Hom}(P_{J^c},P) \otimes_B P_{J^c}).
\]
Then we have $R\text{Hom}(P_{J^c},Q) = 0$.

We will prove that $Q$ is an inverse limit of iterated extensions of the $s_j$ for $j \in J$
by a similar argument to the proof of \cite{NC perv} Theorem 6.2.
It is sufficient to prove that the quotient $Q/\mathfrak m^n Q$ for any $n$ is an iterated extension of 
the $s_j$ for only $j \in J$, where $\mathfrak m$ is the maximal ideal of $R$, because we have
$Q = \varprojlim Q/\mathfrak m^n Q$.
There is a filtration of $Q/\mathfrak m^n Q$ whose quotients are isomorphic to some $s_j$.
We only need to prove that there appear no $s_j$ for $j \in J^c$.
For this purpose, we use $\text{Hom}(P_{J^c}, s_j) = k$ for $j \in J^c$ and $\text{Hom}(P_{J^c}, s_j[1]) = 0$ for all $j$
together with $\text{Hom}(P_{J^c},Q) = 0$.
If there is a quotient which is isomorphic to some $s_j$ with $j \in J^c$, then there is a morphism from $P_{J^c}$ 
to this quotient, which extends successively to $Q/\mathfrak m^n Q$ for all $n$, and we obtain a contradiction
with $\text{Hom}(P_{J^c},Q) = 0$.  

Now in order to prove the versality of $Q$, it is sufficient to prove that $\text{Hom}(Q,s_j) = k$ and 
$\text{Hom}(Q, s_j[1]) = 0$ for $j \in J$ by the argument of \cite{NC perv} Theorem 6.1.
(The first assertion implies that the successive extensions of the $s_j$ for $j \in J$ towards $Q$ are all non-trivial, 
while the second implies that the final extension $Q$ is maximal.)
Any non-zero homomorphism $\text{Hom}(P_{J^c},P) \otimes_B P_{J^c} \to s_j$ is surjective in $\text{Perv}(Y/X)$ 
since $s_j$ is simple.
Then the composed morphism $F_0 \otimes P_{J^c} \to s_j$ is also surjective.
Therefore $\text{Hom}(\text{Hom}(P_{J^c},P) \otimes_B P_{J^c}, s_j) = 0$ for $j \in J$. 
Then it follows that $\text{Hom}(Q,s_j) \cong \text{Hom}(P,s_j) \cong k$ and 
$\text{Hom}(Q, s_j[1]) = 0$ for $j \in J$.

\vskip 1pc

(3) We know that the parameter algebra of the versal deformation is given as $A_{\text{def}} := \text{End}(Q)$ by \cite{NC multi}.
$Q \oplus \bigoplus_{j \in J^c}s_j$ is an NC deformation of a simple collection $\bigoplus_{j=1}^m s_j$.
Since $P$ is a versal NC deformation of the simple collection $\bigoplus_{j=1}^m s_j$, there is 
a homomorphism $h: P \to Q \oplus \bigoplus_{j \in J^c}s_j$ of NC deformations.
We have an associated ring homomorphism of the parameter algebras $h_*: A \to A_{\text{def}}$.

Let $M,M_{\text{def}}$ be respectively the two-sided ideals of $A,A_{\text{def}}$ consisting of elements which induce $0$ maps
on the central fiber $\bigoplus_{j=1}^m s_j$.
We have $A/M \cong A_{\text{def}}/M_{\text{def}} \cong k^m$.
Then the ring homomorphism $h_*$ induces a homomorphism between the Zariski cotangent spaces
$M/M^2 \to M_{\text{def}}/M^2_{\text{def}}$, which is the same as the projection
\[
\bigoplus_{i,j=1}^m \text{Ext}^1(s_i,s_j)^* \to \bigoplus_{i,j \in J} \text{Ext}^1(s_i,s_j)^* 
\]
(cf. \cite{NC formal}).
Since this is surjective, so is the ring homomorphism $h_*$, because 
$A = \varprojlim A/M^n$ and $A_{\text{def}} = \varprojlim A_{\text{def}}/M^n_{\text{def}}$.
It follows that $h$ is also surjective. 

Let $g \in \text{End}(P)$ be an endomorphism which factors through $\bigoplus_{i \in J^c} P_i$:
\[
\begin{CD}
P @>>> \bigoplus_{i \in J^c} P_i @>>> P \\
@VVV @. @VVV \\
Q @>{h_*(g)}>> @. Q
\end{CD}
\]
Since $\text{Hom}(P_i,s_j) = 0$ for $i \in J^c$ and $j \in J$, we have
$\text{Hom}(\bigoplus_{j \in J^c} P_i,Q) = 0$.
Hence $h_*g = 0$.

Conversely, assume that $g \in \text{End}(P)$ satisfies $h_*g = 0$.
We have a commutative diagram of exact sequences in $\text{Perv}(Y/X)$:
\[
\begin{CD}
0 @>>> \text{Hom}(P_{J^c},P) \otimes_B P_{J^c} @>>> P @>>> Q @>>> 0 \\
@. @VVV @VgVV @V0VV \\
0 @>>> \text{Hom}(P_{J^c},P) \otimes_B P_{J^c} @>>> P @>>> Q @>>> 0. 
\end{CD}
\]
By the diagram chasing, there is a homomorphism $\bar g: P \to \text{Hom}(P_{J^c},P) \otimes_B P_{J^c}$
which lifts $g$.
Since $F_0 \otimes_B P_{J^c} \to \text{Hom}(P_{J^c},P) \otimes_B P_{J^c}$ is a surjection in $\text{Perv}(Y/X)$ and $P$ is a
projective object, $\bar g$ is lifted to $\bar g_0: P \to F_0 \otimes_B P_{J^c}$.
Therefore $g$ belongs to the ideal $I$.
\end{proof}

%%%%%%%%%%%%%%%%%%%%%%
%%%%%%%%%%%%%%%%%%%%%%
%%%%%%%%%%%%%%%%%%%%%%
\section{Second infinitesimal neighborhood of a smooth rational curve}

Let $F$ be a coherent sheaf on an algebraic variety $X$ with proper support.
We can describe the NC deformation ring $A_{\text{def}}$ of $F$, the parameter algebra of the versal NC deformation, 
by using $A^{\infty}$ algebra multiplications (\cite{NC formal}).
The tangent space of the deformation ring is given by $\text{Ext}^1(F,F)$, and the 
relations by $\text{Ext}^2(F,F)$ in the following way.

There are $A^{\infty}$-multiplications $m_n: (\text{Ext}^1(F,F))^{\otimes n} \to \text{Ext}^2(F,F)$ for $n \ge 2$.
$m_2$ coincides with the usual composition.
Let 
\[
m = \sum m_n: \hat T(\text{Ext}^1(F,F)) \to \text{Ext}^2(F,F)
\]
be their formal sum, where $\hat T(\text{Ext}^1(F,F)) = \prod_{n=0}^{\infty} (\text{Ext}^1(F,F))^{\otimes n}$ is the 
completed tensor algebra and we put 
$m_0=m_1=0$.
Then the NC deformation ring is given by
\[
A_{\text{def}} = \hat T(\text{Ext}^1(F,F)^*)/m^*(\text{Ext}^2(F,F)^*)
\]
where 
\[
m^*: \text{Ext}^2(F,F)^* \to \hat T(\text{Ext}^1(F,F)^*)
\]
is the dual map.
In particular, it is a quotient algebra of a non-commutative power series ring of dimension $\dim \text{Ext}^1(F,F)$ 
divided by a two-sided ideal generated by $\dim \text{Ext}^2(F,F)$ elements. 

\vskip 1pc

Now we apply the above description to the case where $C \subset X$ is a smooth rational curve on a smooth variety.
We consider NC deformations of a sheaf $\mathcal{O}_C$.

The quadratic terms of the relations are determined by the composition 
\[
m_2: \text{Ext}^1(\mathcal{O}_C,\mathcal{O}_C) \otimes \text{Ext}^1(\mathcal{O}_C,\mathcal{O}_C) 
\to \text{Ext}^2(\mathcal{O}_C,\mathcal{O}_C).
\]
If $m_2$ is skew-symmetric (resp. symmetric), then the quadratic term is a commutative (resp. anti-commutative) relation 
by the following reason.

The quadratic terms of the relations are the image of the homomorphism
\[
m_2^*: \text{Ext}^2(\mathcal{O}_C,\mathcal{O}_C)^* \to 
\text{Ext}^1(\mathcal{O}_C,\mathcal{O}_C)^* \otimes \text{Ext}^1(\mathcal{O}_C,\mathcal{O}_C)^*
\]
which is the dual of $m_2$.
If $m_2(a_i,a_j) = \sum c_{ijk}b_k$, then we have $m_2^*b_k^* = \sum c_{ijk}a_i^* \otimes a_j^*$.
Therefore, if $m_2$ is symmetric (resp. skew-symmetric), i.e., $m_2(a,b) = m_2(b,a)$
(resp. $m_2(a,b) + m_2(b,a) = 0$), 
then the image of $m_2^*$ is generated by elements of the form $a^*b^*+b^*a^*$ (resp. $a^*b^*-b^*a^*$).

\vskip 1pc

If the normal bundle of $C$ has a direct sum decomposition with mixed signs, then the composition is symmetric:

\begin{Thm}\label{symmetric}
Let $X$ be a smooth $3$-dimensional algebraic variety and let $C$ be a subvariety which is isomorphic to $\mathbf{P}^1$ .
Assume that the normal bundle $N_{C/X} \cong \mathcal{O}_C(1) \oplus \mathcal{O}_C(-a)$ for some $a > 0$.
Then a natural bilinear form
\begin{equation}\label{bilinear}
\text{Ext}^1(\mathcal{O}_C,\mathcal{O}_C) \otimes \text{Ext}^1(\mathcal{O}_C,\mathcal{O}_C) 
\to \text{Ext}^2(\mathcal{O}_C,\mathcal{O}_C)
\end{equation}
is symmetric.
\end{Thm}

\begin{proof}
Since $C$ is a locally complete intersection, 
we have $\mathcal{E}xt^p(\mathcal{O}_C,\mathcal{O}_C) \cong \mathcal{O}_C,N_{C/X},\bigwedge^2 N_{C/X}$ for $p=0,1,2$.
We have a spectral sequence
\[
E_2^{p,q} = H^p(\mathcal{E}xt^q(\mathcal{O}_C,\mathcal{O}_C)) \Rightarrow \text{Ext}^{p+q}(\mathcal{O}_C,\mathcal{O}_C).
\]
Since $H^p(\mathcal{O}_C) = 0$ for $p = 1,2$, 
we have $\text{Ext}^1(\mathcal{O}_C,\mathcal{O}_C) \cong H^0(N_{C/X}) \cong H^0(\mathcal{O}_C(1))$.

Let $I_C$ be the ideal sheaf of $C$, i.e., $\mathcal{O}_C = \mathcal{O}_X/I_C$, and let $\mathcal{O}_{2C} = \mathcal{O}_X/I_C^2$.
We have an exact sequence 
\begin{equation}\label{2C}
0 \to N_{C/X}^* \to \mathcal{O}_{2C} \to \mathcal{O}_C \to 0
\end{equation}
and a long exact sequence
\[
\begin{split}
&0 \to \text{Hom}(\mathcal{O}_C,\mathcal{O}_C) \to \text{Hom}(\mathcal{O}_{2C},\mathcal{O}_C) 
\to \text{Hom}(N_{C/X}^*,\mathcal{O}_C) \\
&\to \text{Ext}^1(\mathcal{O}_C,\mathcal{O}_C). 
\end{split}
\]
Since the first two terms are isomorphic to $k$ and the next two terms have the same dimension, 
the connecting homomorphism
$\text{Hom}(N_{C/X}^*,\mathcal{O}_C) \to \text{Ext}^1(\mathcal{O}_C,\mathcal{O}_C)$
is a bijection, which is given by the extension class $e \in \text{Ext}^1(\mathcal{O}_C,N_{C/X}^*)$ of (\ref{2C}) and
a composition map
\[
\text{Ext}^1(\mathcal{O}_C,N_{C/X}^*) \times \text{Hom}(N_{C/X}^*,\mathcal{O}_C) 
\to \text{Ext}^1(\mathcal{O}_C,\mathcal{O}_C)
\]
i.e., $h \in \text{Hom}(N_{C/X}^*,\mathcal{O}_C)$ is mapped to $g = h[1]e \in \text{Ext}^1(\mathcal{O}_C,\mathcal{O}_C)$.

Let $p: N_{C/X}^* \to \mathcal{O}_C(-1)$ be the projection to a direct summand.
Then we can further write $g = h_1[1]p[1]e$ for $h_1 \in \text{Hom}(\mathcal{O}_C(-1),\mathcal{O}_C)$.
The extension $\mathcal{O}_{C_2}$ of $\mathcal{O}_C$ by $\mathcal{O}_C(-1)$ 
corresponding to $p[1]e \in \text{Ext}^1(\mathcal{O}_C,\mathcal{O}_C(-1))$ is obtained by the 
following commutative diagram
\[
\begin{CD}
0 @>>> N_{C/X}^* @>>> \mathcal{O}_{2C} @>>> \mathcal{O}_C @>>> 0 \\
@. @VpVV @VVV @V=VV \\
0 @>>> \mathcal{O}_C(-1) @>>> \mathcal{O}_{C_2} @>>> \mathcal{O}_C @>>> 0.
\end{CD}
\]

Let $g,g' \in \text{Ext}^1(\mathcal{O}_C,\mathcal{O}_C)$.
Then our bilinear form is given by $(g,g') \mapsto g[1]g' \in \text{Hom}(\mathcal{O}_C,\mathcal{O}_C[2]) 
= \text{Ext}^2(\mathcal{O}_C,\mathcal{O}_C)$. 
We write $g = h_1[1]p[1]e$ and $g' = h'_1[1]p[1]e$.
In order to prove that $gg' = g'g$, it is sufficient to prove that 
$h_1[1]p[1]eh'_1 = h'_1[1]p[1]eh_1 \in \text{Hom}(\mathcal{O}_C(-1),\mathcal{O}_C[1])$.

Let $G,G'$ be the extensions of $\mathcal{O}_C(-1)$ by $\mathcal{O}_C$ corresponding to $h_1[1]p[1]eh'_1$ and 
$h'_1[1]p[1]eh_1$.
We will prove that they are isomorphic as extensions.
$G$ is obtained by the following commutative diagram:
\[
\begin{CD}
0 @>>> \mathcal{O}_C(-1) @>>> \mathcal{O}_{C_2} @>>> \mathcal{O}_C @>>> 0 \\
@. @A=AA @AAA @A{h'_1}AA \\
0 @>>> \mathcal{O}_C(-1) @>>> F @>>> \mathcal{O}_C(-1) @>>> 0 \\
@. @V{h_1}VV @VVV @V=VV \\
0 @>>> \mathcal{O}_C @>>> G @>>> \mathcal{O}_C(-1) @>>> 0.
\end{CD}
\]
$G'$ is similarly defined by interchanging $h_1$ and $h'_1$.

We need to prove that $G' \cong G$ as extensions of $\mathcal{O}_C(-1)$ by $\mathcal{O}_C$.
We note that the $\mathcal{O}_X$-submodule $\mathcal{O}_C \subset G$ is characterized by
\[
\Gamma(U,\mathcal{O}_C) = \{s \in \Gamma(U,G) \mid I_Cs = 0\}
\]
for any open subset $U \subset C$, and similarly for $G'$.
Therefore it is sufficient to prove that $G' \cong G$ as $\mathcal{O}_X$-modules.

Let $P,P'$ be the supports of the cokernels of the injective homomorphisms $h_1, h'_1$, respectively.
Let $L$ be any invertible sheaf on $C_2 \subset 2C$ such that 
$L \otimes \mathcal{O}_C = \mathcal{O}_C(P' - P) \cong \mathcal{O}_C$.
We have $L \cong \mathcal{O}_{C_2}$ because $H^1(\mathcal{O}_C(-1)) = 0$.
Indeed we have an exact sequence
\[
H^1(\mathcal{O}_C(-1)) \to H^1(\mathcal{O}_{C_2}^*) \to H^1(\mathcal{O}_C^*) \to 0.
\]

Locally around $P$ and $P'$, we take analytic local coordinates $(x,y,z)$ and $(x',y',z')$ respectively, which satisfy the 
following: 
\[
\begin{split}
&\mathcal{O}_{C,P} = k\{x,y,z\}/(y,z) \cong k\{x\}, \,\,\, \mathcal{O}_{C,P'} = k\{x',y',z'\}/(y',z') \cong k\{x'\}, \\
&\mathcal{O}_{C_2,P} = k\{x,y,z\}/(y^2,z) \cong k\{x,y\}/(y^2), \\
&\mathcal{O}_{C_2,P'} = k\{x',y',z'\}/(y^{\prime 2},z') 
\cong k\{x',y'\}/(y^{\prime 2}), \\
&h_1\mathcal{O}_{C,P}(-1) = x(k\{x,y,z\}/(y,z)) \cong xk\{x\}, \\
&h_1'\mathcal{O}_{C,P'}(-1) = x'(k\{x',y',z'\}/(y',z')) \cong x'k\{x'\}.
\end{split}
\]
Then 
\[
G_{P'} = (x'+p'(y'),y')k\{x',y'\}/(y^{\prime 2}) = (x',y')k\{x',y'\}/(y^{\prime 2})
\]
for some function $p'(y') \in y'k\{x',y'\}$, and 
\[
G_P = (1,x^{-1}y)k\{x,y\}/(x^{-1}y^2)
\]
where $z,z'$ act trivially on these modules.
On the other hand, we have $G'_P = (x,y)k\{x,y\}/(y^2)$ and 
$G'_{P'} = (1,x^{\prime -1}y')k\{x',y'\}/(x^{\prime -1}y^{\prime 2})$.
Therefore we have $G' \cong G \otimes L \cong G$.
\end{proof}

On the other hand, if the normal bundle is positive, then the composition is skew-symmetric:

\begin{Prop}\label{skew-symmetric}
Let $X$ be a smooth algebraic variety and $C \cong \mathbf{P}^1$ a subvariety.
Assume that $H^1(N_{C/X}) = 0$.
Then the bilinear form (\ref{bilinear}) is skew-symmetric.
\end{Prop}

\begin{proof}
We have $\text{Ext}^1(\mathcal{O}_C,\mathcal{O}_C) = H^0(N_{C/X})$, 
$\text{Ext}^2(\mathcal{O}_C,\mathcal{O}_C) = H^0(\bigwedge^2 N_{C/X})$, and the 
bilinear form (\ref{bilinear}) comes from the wedge product.
Therefore it is skew-symmetric.
\end{proof}

%%%%%%%%%%%%%%%%%%%
%%%%%%%%%%%%%%%%%%%
%%%%%%%%%%%%%%%%%%%
\section{Example: universal flopping contraction of length $2$}

A projective birational morphism $f: Y \to X$ from a variety with only terminal (or canonical, or more) singularities to a normal variety
is called a {\em flopping contraction} if the exceptional locus
has codimension at least $2$ and the canonical divisor $K_Y$ is numerically trivial along the exceptional curves.
We consider only the case where $Y$ is smooth in this paper.

Let us consider the case where $\dim Y = 3$, the exceptional curve is isomorphic to $\mathbf{P}^1$, and $X$ is a germ of a singularity.  
The analytic type of a generic hyperplane section of $X$ through its singularity is classified by Katz and Morrison \cite{KatzM}
using a result of \cite{Reid} (an easy alternative proof is found in \cite{hyperplane}) in the case where the exceptional locus is irreducible:

\begin{Thm}  
Let $f: Y \to X$ be a flopping contraction of a smooth $3$-fold such that 
the exceptional locus $C$ of $f$ is a smooth rational curve.
Let $P = f(C)$ be the singular point of $X$, let $H$ be a general hyperplane section of $X$ through $P$, 
and let $L = f^{-1}(H)$ be its inverse image.
Let $l$ be the length of the scheme theoretic fiber $f^{-1}(P)$ at the generic point of $C$.
Then the length $l$ takes value in a set $\{1,2,3,4,5,6\}$, and 
the singularity of $H$ together with the partial resolution $f_H: L \to H$ is determined by $l$.
More precisely,
$H$ has a rational double point of type $A_1, D_4, E_6, E_7, E_8$, or $E_8$,
if $l = 1, 2, 3, 4, 5$ or $6$, respectively, and the exceptional divisor of $f_H$ is the 
rational curve which appears in the minimal resolution of $H$ and uniquely determined by the condition that 
it has multiplicity $l$ in the fundamental cycle.
\end{Thm}

A {\em universal flopping contraction} morphism $\tilde f: \tilde Y \to \tilde X$ of length $l$ 
is a versal deformation of the contraction morphism
$f_H: L \to H$ of surfaces with the given length as a neighborhood of the exceptional curve $C$
described in the above theorem.
We have $\dim \tilde X = 3,6,8,9,10,10$ if $l = 1,2,3,4,5,6$ (cf. \cite{Karmazyn}).

As a corollary of the above theorem, we deduce that a universal flopping contraction morphism is universal in the following sense: 
for any flopping contraction $f: Y \to X$ of smooth $3$-fold with irreducible exceptional locus and length $l$, there
is a morphism $p: X \to \tilde X$ such that $f$ is isomorphic to the pull back of $\tilde f$ by $p$, so that 
$Y \cong \tilde Y \times_{\tilde X} X$ and $f$ corresponds to the second projection. 

Curto-Morrison \cite{Curto-Morrison} constructed a universal flopping contraction morphism of length $l = 2$ explicitly:

\begin{Thm}\label{Curto-Morrison Thm}
Let $\tilde X$ is a hypersurface given by the following equation in $k^7$:
\begin{equation}\label{universal flop}
F = x^2 + uy^2 + 2vyz + wz^2 + (uw-v^2)t^2 = 0.
\end{equation}
Then there exists a maximally Cohen-Macaulay sheaf $f_*M$ of rank $2$ on $\tilde X$ such that a 
universal flopping contraction morphism of length $l = 2$ is given as a {\em Grassmann blowup} $\tilde f: \tilde Y \to \tilde X$, 
a universal projective birational morphism such that the inverse image modulo torsion $M$ of $f_*M$ become locally free.
\end{Thm}

Let $S = k[x,y,z,t,u,v,w]$ be a polynomial ring and let $R = S/(F) = \mathcal{O}_{\tilde X}$.
The sheaf $f_*M$ has a matrix factorization (\cite{Eisenbud}) as follows (\cite{Curto-Morrison}):
it has a resolution by free $S$-modules
\[
\begin{CD}
\dots @>{\Psi}>> S^4 @>{\Phi}>> S^4 @>{\Psi}>> S^4 @>{\Phi}>> S^4 @>>> f_*M @>>> 0 
\end{CD}
\]
where $\Phi=xI-\Xi$ and $\Psi=xI+\Xi$ with
\[
\Xi = \left( \begin{matrix}
-vt & y & z & t \\
-uy-2vz & vt  & -ut & z \\
-wz & wt & -vt & -y \\
-uwt & -wz & uy+2vz & vt
\end{matrix} \right)
\]
such that 
\[
\Phi \Psi = \Psi \Phi = FI_4.
\]

$f_*M$ and $R = \mathcal{O}_{\tilde X}$ are the indecomposable maximally Cohen-Macaulay sheaves on $\tilde X$.
According to Van den Bergh (\cite{VdBergh}), 
$M \oplus \mathcal{O}_{\tilde Y}$ become a tilting generator of $D(Q\text{coh}(\tilde Y))$, and the category of the 
perverse coherent sheaves ${}^{-1}\text{Perv}(\tilde Y/\tilde X)$
(denoted $\text{Perv}(\tilde Y/\tilde X)$ in \S2) is defined as the subcategory of $D^b(\text{coh}(\tilde Y))$ 
which corresponds to the category of modules $(\text{mod-}A)$ under the Bondal-Rickard equivalence
\[
D^b(\text{coh}(\tilde Y)) \cong D^b(\text{mod-}A)
\]
where $A = \text{End}(M \oplus \mathcal{O}_{\tilde Y})$ is a sheaf of associative algebras on $\tilde X$

\vskip 1pc

The algebra $A$ is determined by Aspinwall-Morrison \cite{AM}.
There are generators: $a,b: f_*M \to f_*M$, $c: f_*M \to R$, $d: R \to f_*M$
expressed by using the matrix factorization: we have 
\[
\begin{split}
&a = \left( \begin{matrix} 
0 & 1 & 0 & 0 \\
-u & 0 & 0 & 0 \\
-2v & 0 & 0 & 1 \\
0 & 2v & -u & 0 \end{matrix} \right), \quad 
b =  \left( \begin{matrix} 
0 & 0 & 1 & 0 \\
0 & 0 & 0 & -1 \\
-w & 0 & 0 & 0 \\
0 & w & 0 & 0 \end{matrix} \right), \\
&c = \left( \begin{matrix} x-vt & y & z & t \end{matrix} \right), \quad 
d = \left( \begin{matrix} 
0 \\
0 \\
0 \\
1 \end{matrix} \right).
\end{split}
\]
For example, we have a commutative diagram
\[
\begin{CD}
\dots @>{\Psi}>> S^4 @>{\Phi}>> S^4 @>>> f_*M @>>> 0 \\
@. @VaVV @VaVV @VaVV @. \\
\dots @>{\Psi}>> S^4 @>{\Phi}>> S^4 @>>> f_*M @>>> 0
\end{CD}
\]
where we used the same symbol $a$ for an element of $A$ and its lift given by a matrix.

\begin{Thm}\label{normal bundle}
The normal bundle of the reduced fiber $C = \tilde f^{-1}(0)_{\text{red}}$ is given by 
\[
N_{C/\tilde X} \cong \mathcal{O}_C(1) \oplus \mathcal{O}_C \oplus \mathcal{O}_C(-1)^3.
\]
\end{Thm}

\begin{proof}
By \cite{Curto-Morrison}, a neighborhood of $C \subset \tilde Y$ is covered by two open subsets $U_1, U_2$. 
$U_1$ has coordinates $(z,t,u,v,\alpha_{12}, \alpha_{22})$ such that 
the coordinates $x,y,w$ of $\tilde X$ are given by (\cite{Curto-Morrison}~formulas (44), (45), (48)):
\[
\begin{split}
&y + \alpha_{12}z + \alpha_{22}t = 0, \\
&w + \alpha_{22}^2 + \alpha_{12}^2u - 2\alpha_{12}v = 0, \\
&x+vt - \alpha_{12}ut + \alpha_{22}z = 0, 
\end{split}
\]
while $U_2$ has coordinates $(y,t,v,w,\beta_{12}, \beta_{22})$ such that (\cite{Curto-Morrison}~formula (57)):
\[
\begin{split}
&\beta_{12}y + z + \beta_{22}t = 0, \\
&\beta_{12}^2w + \beta_{22}^2 + u - 2\beta_{12}v = 0, \\
&x - vt + \beta_{12}wt - \beta_{22}y = 0.
\end{split}
\]
These formulas are equivalent under the transformation: 
\[
\alpha_{12}\beta_{12} = 1, \quad \alpha_{22}\beta_{12} = \beta_{22}.
\]

The conormal bundle of the reduced central fiber $C$ is generated by $(z,t,u,v,\alpha_{22})$ in $U_1$ 
and $(y,t,v,w,\beta_{22})$ in $U_2$.
Since we have 
\[
y \equiv - \alpha_{12}z \mod I_C^2
\] 
there is a subbundle $L_1$ of degree $1$ of $N^*_{C/\tilde Y}$ generated by $z$ in $U_1$ and by $y$ in $U_2$.
We have also 
\[
w \equiv - \alpha_{12}^2u + 2\alpha_{12}v = \alpha_{12}(- \alpha_{12}u + 2v) \mod I_C^2
\]
and there is a subbundle $L_2$ of degree $1$ generated by $- \alpha_{12}u + 2v$ in $U_1$ and by $w$ in $U_2$.
A relation 
\[
2v \equiv \alpha_{12}u \mod L_2
\]
gives a subbundle $L_3$ of degree $1$ generated by $u$ in $U_1$ and by $2v$ in $U_2$.

$t$ generates a subbundle $L_4$ of degree $0$.
There is a subbundle $L_5$ of degree $-1$ generated by $\alpha_{22}$ in $U_1$ and by $\beta_{22}$ in $U_2$.
Therefore we have our claim.
\end{proof}

\begin{Rem}
We have $H^1(N_{C/\tilde Y}) = 0$.
Hence the commutative deformations of $C$ in $\tilde Y$ have no obstruction.

We have $\dim \text{Ext}^1(\mathcal{O}_C,\mathcal{O}_C) = \dim H^0(N_{C/\tilde X}) = 3$, hence
the tangent spaces of commutative and 
non-commutative deformations are the same.
We have a basis $\{a,b,t\}$ of the cotangent space
$H^0(N_{C/\tilde X})^* = \text{Ext}^1(\mathcal{O}_C,\mathcal{O}_C)^*$.
On the other hand, we have 
\[
\bigwedge^2 N_{C/\tilde X} \cong \mathcal{O}_C(1) \oplus \mathcal{O}_C^3 \oplus \mathcal{O}_C(-1)^3 \oplus \mathcal{O}_C(-2)^3.
\]
Hence $\dim \text{Ext}^2(\mathcal{O}_C,\mathcal{O}_C) = \dim H^0(\bigwedge^2 N_{C/\tilde X}) = 5$, and
there are $5$ relations among the generators $a,b,t$ in NC deformations.
\end{Rem}

First we determine the NC deformation algebra for the reduced fiber $C = \tilde f^{-1}(0)_{\text{red}}$:

\begin{Thm}\label{univ def alg length 2}
The NC deformation algebra of $\mathcal{O}_C$ on $\tilde Y$ is given by
\[
\begin{split}
&A^1_{\text{def}} = k\langle \langle a,b,t \rangle \rangle/(at-ta,bt-tb,tab-tba, ab^2-b^2a, a^2b-ba^2) \\
&= k[[t]] \langle \langle a,b\rangle \rangle/(tab-tba, ab^2-b^2a, a^2b-ba^2).
\end{split}
\]
\end{Thm}

\begin{proof}
The set of simple objects in ${}^{-1}\text{Perv}(\tilde Y/\tilde X)$ corresponding to the set of indecomposable
projective objects $\{M,\mathcal{O}_{\tilde Y}\}$ is $\{\mathcal{O}_C(-1)[1], \mathcal{O}_{\tilde C}\}$, where $\tilde C$ is the scheme theoretic fiber
$\tilde f^{-1}(0)$ which has length $2$.
In other words, we have 
\[
\begin{split}
\text{Hom}(M,\mathcal{O}_C(-1)[1]) \cong k, \quad \text{Hom}(M,\mathcal{O}_{\tilde C}) \cong 0 \\
\text{Hom}(\mathcal{O}_{\tilde Y},\mathcal{O}_C(-1)[1]) \cong 0, \quad \text{Hom}(\mathcal{O}_{\tilde Y},\mathcal{O}_{\tilde C}) \cong k.
\end{split}
\]
Therefore we have $A^1_{\text{def}} = A/I^1$ by Theorem~\ref{partial}, where $I^1$ is a two-sided ideal generated by endomorphisms 
which are compositions of $2$ homomorphisms of the form 
$M \oplus \mathcal{O}_{\tilde Y} \to \mathcal{O}_{\tilde Y} \to M \oplus \mathcal{O}_{\tilde Y}$.
Therefore $A^1_{\text{def}}$ is generated by $a,b$ over $R$.

We can check the following equations by examining the matrices as endomorphisms of $S^4$:
\begin{equation}\label{u,v,w}
\begin{split}
&a^2 = - u \\
&b^2 = - w \\
&ab+ba = -2v.
\end{split}
\end{equation}

We have 
\[
\begin{split}
&y - tb + bdc + dcb \\
&= \left( \begin{matrix} 
y & 0 & -t & 0 \\
0 & y & 0 & t \\
wt & 0 & y & 0 \\
0 & -wt & 0 & y 
\end{matrix} \right)
+ \left( \begin{matrix} 
0 & 0 & 0 & 0 \\
-x+vt & -y & -z & -t \\
0 & 0 & 0 & 0 \\
0 & 0 & 0 & 0 
\end{matrix} \right)
+ \left( \begin{matrix} 
0 & 0 & 0 & 0 \\
0 & 0 & 0 & 0 \\
0 & 0 & 0 & 0 \\
-wz & wt & x-vt & -y 
\end{matrix} \right) \\
&= \left( \begin{matrix} 
y & 0 & -t & 0 \\
-x+vt & 0 & -z & 0 \\ 
wt & 0 & y & 0 \\
-wz & 0 & x-vt & 0
\end{matrix} \right) \equiv 0 \mod \text{Im}(\Phi)
\end{split}
\]
as endomorphisms of $S^4$, so that we have
\[
y - tb + bdc + dcb = 0
\]
as endomorphisms of $f_*M = \text{Coker}(\Phi)$, hence 
\begin{equation}\label{y}
y = tb
\end{equation}
in $A^1_{\text{def}}$.
We also have
\[
\begin{split}
&z + ta - adc - dca \\
&= \left( \begin{matrix} 
z & t & 0 & 0 \\
-ut & z & 0 & 0 \\ 
-2vt & 0 & z & t \\
0 & 2vt & -ut & z
\end{matrix} \right) 
- \left( \begin{matrix} 
0 & 0 & 0 & 0 \\
0 & 0 & 0 & 0 \\ 
x-vt & y & z & t \\
0 & 0 & 0 & 0
\end{matrix} \right) 
- \left( \begin{matrix} 
0 & 0 & 0 & 0 \\
0 & 0 & 0 & 0 \\ 
0 & 0 & 0 & 0 \\
-uy-2vz & x+vt & -ut & z
\end{matrix} \right) \\
&= \left( \begin{matrix} 
z & t & 0 & 0 \\
-ut & z & 0 & 0 \\ 
-x-vt & -y & 0 & 0 \\
uy+2vz & -x+vt & 0 & 0
\end{matrix} \right) \equiv 0 \mod \text{Im}(\Phi) \\
&x + vt + tba - dcab + badc \\
&= \left( \begin{matrix} 
x-vt & 0 & 0 & t \\
0 & x-vt & ut & 0 \\ 
0 & -wt & x+vt & 0 \\
-uwt & 0 & 0 & x+vt
\end{matrix} \right)
+ \left( \begin{matrix} 
0 & 0 & 0 & 0 \\
0 & 0 & 0 & 0 \\ 
0 & 0 & 0 & 0 \\
uwt & wz & -uy-2vz & -x-vt
\end{matrix} \right) \\
&- \left( \begin{matrix} 
x-vt & y & z & t \\
0 & 0 & 0 & 0 \\ 
0 & 0 & 0 & 0 \\
0 & 0 & 0 & 0
\end{matrix} \right) 
= \left( \begin{matrix} 
0 & -y & -z & 0 \\
0 & x-vt & ut & 0 \\ 
0 & -wt & x+vt & 0 \\
0 & wz & -uy-2vz & 0
\end{matrix} \right) \equiv 0 \mod \text{Im}(\Phi).
\end{split}
\]
Hence
\[
\begin{split}
&z = - ta \\
&x = - tba - vt
\end{split}
\]
in $A^1_{\text{def}}$.
Therefore $A^1_{\text{def}}$ is generated by $a,b,t$ over $k$.

We determine quadratic relations.
Since $H^1(N_{C/\tilde X}) = 0$, they are skew-symmetric bilinear form (Proposition~\ref{skew-symmetric}).
Since $t \in R$ is in the center, we have $ta-at=0$ and $tb-bt=0$, and there are no more quadratic terms.
Since there are $5$ generators of the relation ideal, there are $3$ more relations.

Since $ay=ya$, we have $tab = atb = tba$, hence $tab-tba=0$.
Since $a^2,b^2$ are in the center, we have 
$ab^2-b^2a = a^2b-ba^2 = 0$.
These $3$ relations are order $3$ and linearly independent.
Therefore there are no more relations.
\end{proof}

The commutative deformations of $\mathcal{O}_C$ on $\tilde Y$ is given by the following.
We denote by $(A^1_{\text{def}})^{ab}$ the abelianization of the NC deformation algebra 
$A^1_{\text{def}}$, which is the parameter algebra of commutative deformations.
(This is because the natural homomorphism $A^1_{\text{def}} \to (A^1_{\text{def}})^{ab}$ is universal
among local homomorphisms $A^1_{\text{def}} \to B$ to Artin local algebras). 

\begin{Cor}
$(A^1_{\text{def}})^{ab} = k[[a,b,t]]$. 
\end{Cor}

Next we investigate NC deformations of the scheme theoretic fiber $\tilde C = \tilde f^{-1}(0)$.

\begin{Thm}\label{univ def alg length 2-2}
The NC deformation algebra of $\mathcal{O}_{\tilde C}$ on $\tilde Y$ is given by
\[
A^2_{\text{def}} = k[[u,v,w]].
\] 
\end{Thm}

\begin{proof}
We have $A^2_{\text{def}} = A/I^2$ by Theorem~\ref{partial}, where $I^2$ is a two-sided ideal generated by endomorphisms 
which are compositions of $2$ homomorphisms of the form 
$M \oplus \mathcal{O}_{\tilde Y} \to M \to M \oplus \mathcal{O}_{\tilde Y}$.
Therefore $A^2_{\text{def}}$ is a quotient ring of $R$.
In particular, there are only commutative deformations.

We have 
\[
cd = t, \quad cad = z, \quad cbd = -y, \quad cbad = x - vt.
\]
Therefore $A^2_{\text{def}}$ is a quotient ring of $k[[u,v,w]]$.

There is an affine space $D$ of dimension $3$ with coordinates $(u,v,w)$ defined by $x=y=z=t=0$ 
contained in $\tilde X$.
If we pull it back by $\tilde f$, we obtain a $4$-dimensional smooth subspace $E$ of $\tilde Y$  
with coordinates $(u,v,\alpha_{12},\alpha_{22})$ on $U_1$ and $(v,w,\beta_{12},\beta_{22})$ on $U_2$, where
we used the notation of the proof of Theorem~\ref{normal bundle}.
The morphism $\tilde f_E: E \to D$ is given by
\[
\begin{split}
&w = - \alpha_{22}^2 - \alpha_{12}^2u + 2\alpha_{12}v \\
&u = - \beta_{22}^2 - \beta_{12}^2w + 2\beta_{12}v.
\end{split}
\]
Thus $\tilde f_E$ is a flat morphism and gives a commutative deformation of the fiber $\tilde C$
with the whole space $D$ as a parameter space. 
By taking the completion at the origin, we obtain the formal power series ring as stated.
\end{proof}

We determine the singular locus of $\tilde X$ as well as the singularities of along it:

\begin{Lem}
(1) The singular locus of $\tilde X$ consists of two irreducible components 
$\text{Sing}(\tilde X) = S_1 \cup S_2$ with 
\[
\begin{split}
&S_1 = \{x = uw - v^2 = uy+vz = y^2 + wt^2 = z^2+ut^2 = 0\}, \\
&S_2 = \{x=y=z=t=0\}.
\end{split}
\] 

(2) $\tilde X$ has $A_1$ singularities along generic points of the singular locus 
$\text{Sing}(\tilde X) = S_1 \cup S_2$.
\end{Lem}

\begin{proof}
(1) The partial derivations of $F$ yield equations of the singular locus
\[
\begin{split}
&x = uy + vz = vy + wz = t(uw - v^2) \\
&= y^2 + wt^2 = yz - vt^2 = z^2 + ut^2 = 0.
\end{split}
\]
If $t = 0$, then $x=y=z=t=0$, thus we have $S_2$.
If $t \ne 0$ and $u = 0$, then $x = z = u = v = y^2+wt^2 = 0$. 
This locus is denoted by $S_3$.
If $t \ne 0$ and $w = 0$, then $x = y = v = w = z^2+ut^2=0$.
This locus is denoted by $S_4$

If $uwt \ne 0$, then $vyz \ne 0$, and $y = - vz/u = -wz/v$.
Hence $x = 0$, $uw = v^2$, $uy+vz=0$, $y^2+wt^2=0$, $z^2+ut^2=0$, 
thus we obtain $S_1$ which contains $S_3,S_4$.

(2) Let $A = uw - v^2$, $B = uy+vz$, $C = z^2+ut^2$.
Then we have
\[
uF = ux^2 + B^2 + AC.
\]
Thus we have a family of $A_1$ along generic points of $S_1$.

Since $F$ is a quadratic form on $x,y,z,t$, $\tilde X$ has also a family of $A_1$ along generic points of $S_2$.
\end{proof}

\begin{Prop}
The versal commutative deformation of $\mathcal{O}_C$ (resp. $\mathcal{O}_{\tilde C}$) is 
along the component $S_1$ (resp. $S_2$).
\end{Prop}

\begin{proof}
We have the following equalities in the abelianization $(A^1_{\text{def}})^{ab}$, the parameter algebra of 
the versal commutative deformation:
\[
\begin{split}
&x = -tba + \frac 12 t(ab+ba) = \frac 12 t(ab-ba)=0, \\
&uw-v^2 = a^2b^2 - \frac 14 (ab+ba)^2 = - \frac 14 (ab-ba)^2 = 0, \\
&uy + vz = - ta^2b + \frac 12 t(aba + baa) = \frac 12 t(aba - ba^2) = 0, \\
&z^2+ut^2 = t^2a^2-t^2a^2=0.
\end{split}
\]
Since $\dim S_1 = 3$, the deformation of $\mathcal{O}_C$ covers the whole $S_1$.

The statement for the versal deformation of $\mathcal{O}_{\tilde C}$ is already proved 
in Theorem~\ref{univ def alg length 2-2}.
\end{proof}

\begin{Rem}
The genus zero Gopakumar-Vafa invariants $n_j$ for $1 \le j \le l$ defined by Katz \cite{Katz} 
counts the number of rational curves on deformations of a flopping contraction of a $3$-fold.
In our case of the universal flopping contraction, 
these numbers come from the rational curves above the components $S_j$.
Toda \cite{Toda}~Theorem 1.1 proved a formula connecting the intersection multiplicities of $S_j$ and the dimensions of the NC deformation algebras.
\end{Rem}

%%%%%%%%%%%%%%%%%%%
%%%%%%%%%%%%%%%%%%%
%%%%%%%%%%%%%%%%%%%
\section{Example: deformations of Laufer's flopping contraction}

We consider a family of hypersurfaces $X_{\lambda} \subset k^4$ defined by equations: 
\[
F_{\lambda} = 
x^2 + y^3 + \sum_{i=1}^{2n} \lambda_i y^2(-w)^i + z^2w + yw^{2n+1} - \sum_{i=1}^{2n} \lambda_i(-w)^{i+2n+1} 
= 0
\]
where $n$ is a positive integer, $\lambda = (\lambda_i)$, and $\lambda_i \in k = \mathbf{C}$.
They are obtained from the universal flopping contraction of Curto-Morrison (\ref{universal flop}) by
the substitution:
\begin{equation}\label{Laufer substitution}
t=(-w)^n, \quad u = y + \sum_{i=1}^{2n}\lambda_i(-w)^i, \quad v=0.
\end{equation}
The example of Morrison-Pinkham (\cite{P}) is the case where $n = 1$.
Laufer's example (\cite{Laufer}) of a flopping contraction of length $2$ 
is constructed from $X_0$, and $F_0$ is weighted homogeneous with weights
\[
wt(x,y,z,w) = (6n+3, 4n+2, 6n+1, 4).
\]
We assume that the parameters $\lambda_i$ are small in the sense that $\lambda \in U \subset k^{2n}$ for a neighborhood
$U$ of $0$, but Theorem~\ref{stratification} below implies that $\lambda \in k^{2n}$ can be arbitrary indeed.

\begin{Lem}
Let $\tilde f: \tilde Y \to \tilde X$ be the Grassmann blowup in Theorem~\ref{Curto-Morrison Thm}.
Then the morphism $f_{\lambda}: Y_{\lambda} \to X_{\lambda}$ obtained by the pull back 
by a morphism (\ref{Laufer substitution}) is a flopping contraction of length $2$.
\end{Lem}

\begin{proof}
By construction, $\tilde f$ is a birational morphism which is an isomorphism above the smooth locus of $\tilde X$, the
fibers of $\tilde f$ above the singular locus are $1$-dimensional, and the canonical divisor 
$K_{\tilde Y}$ is relatively numerically trivial.

When $\lambda = 0$, $Y_0 \to X_0$ is Laufer's flopping contraction.
Thus the image of $X_0$ on $\tilde X$ intersects the singular locus only at $0$.
Therefore the only singularity of $X_{\lambda}$ is isolated at $0$ and 
$Y_{\lambda}$ is smooth because they are small deformations of $Y_0 \to X_0$.
$K_{Y_{\lambda}/X_{\lambda}}$ is relatively numerically trivial because it is the pull-back of $K_{\tilde Y/\tilde X}$.
The scheme theoretic fibers coincide $f_{\lambda}^{-1}(0) = f_0^{-1}(0) = \tilde f^{-1}(0)$ as schemes, hence the length is $2$.
\end{proof}

We note that the exceptional locus of the contraction morphism $Y_{\lambda} \to X_{\lambda}$ is 
always the same curve $C$ inside $\tilde Y$.

The deformations of the hypersurface $X_0$ is parametrized by a quotient ring $U_0 = k[x,y,z,w]/J_0$ 
where $J_0$ is an ideal generated by the partial derivatives:
\[
x, \quad zw, \quad 3y^2 + w^{2n+1}, \quad z^2 + (2n+1)yw^{2n}
\]
because $F_0$ is weighted homogeneous.
Then we have $\mod J_0$:
\[
\begin{split}
&x \equiv zw \equiv 0, \quad z^3 \equiv -(2n+1) yzw^{2n} \equiv 0, \\
&yw^{2n+1} \equiv - \frac 1{2n+1} z^2w \equiv 0, \quad y^3 \equiv - \frac 13 yw^{2n+1} \equiv 0, \\
&w^{4n+2} \equiv 9y^4 \equiv 0, \quad y^2z \equiv - \frac 13 zw^{2n+1} \equiv 0, \\
&yz^2 \equiv - (2n+1) y^2w^{2n} \equiv \frac {2n+1}3 w^{4n+1}, \quad y^2 \equiv - \frac 13 w^{2n+1}.
\end{split}
\]
Thus the deformation space $U_0$ for $f_0$ is generated by the following monomials:
\begin{equation}\label{deformation basis}
\begin{split}
&1, \quad w, \quad \dots, \quad w^{4n+1}, \quad z, \quad z^2 \\
&y, \quad yw, \quad \dots, \quad yw^{2n}, \quad yz. 
\end{split}
\end{equation}
If we discard monomials of degree $\le 2$ and those smaller than a monomial $yw^{2n+1}$,
then the remaining monomials are $w^{2n+2},\dots,w^{4n+1}$, 
which are equivalent to $y^2w,\dots,y^2w^{2n} \mod J_0$.
In this way we obtain our $2n$-dimensional deformation family.

\begin{Thm}\label{stratification}
Define a stratification $\{\Sigma_i\}_{i=0}^{2n}$ of the affine $\lambda$-space $k^{2n}$ 
by $\Sigma_0 = \{0\}$ and
$\Sigma_i = \{\lambda \mid \min \{j \mid \lambda_j \ne 0\} = i\}$ for $i > 0$.
Then $X_{\lambda} \cong X_{\lambda'}$ if $\lambda, \lambda' \in \Sigma_i$ for fixed $i$.
\end{Thm}

\begin{proof}
We take a $\lambda \in \Sigma_i$.
The versal deformations of the hypersurface $X_{\lambda}$ is parametrized by a quotient ring 
$U_{\lambda} = k[x,y,z,w]/J_{\lambda}$, where 
$J_{\lambda}$ is an ideal generated by $F_{\lambda}$ and the partial derivatives of $F_{\lambda}$:
\[
\begin{split}
&x, \quad zw, \quad 
3y^2 + \sum_{j=1}^{2n} 2\lambda_j y(-w)^j + w^{2n+1}, \\
&z^2 + (2n+1)yw^{2n} - \sum_{j=1}^{2n} \lambda_j y^2(-w)^{j-1} + \sum_{j=1}^{2n} \lambda_j (-w)^{j+2n}.
\end{split}
\]
We have 
\[
\begin{split}
&(6n+3)x\partial F_{\lambda}/\partial x + (4n+2)y\partial F_{\lambda}/\partial y 
+ (6n+1)z\partial F_{\lambda}/\partial z 
+ 4w\partial F_{\lambda}/\partial w \\
&= (12n+6)F_{\lambda} + \sum_{j=i}^{2n} (4j - 4n - 2) \lambda_j (y^2(-w)^j - (-w)^{j+2n+1}).
\end{split}
\]
Therefore $\sum_{j=i}^{2n} (4j - 4n - 2) \lambda_j (y^2(-w)^j - (-w)^{j+2n+1}) \in J_{\lambda}$ 
with $\lambda_i \ne 0$.
Then it follows that $y^2(-w)^j - (-w)^{j+2n+1} \in J_{\lambda}$ for all $j \ge i$.
Indeed, since $X_{\lambda}$ has an isolated singularity, the ring $U_{\lambda}$ is an Artin local ring.
Then $\sum_{j=i}^{2n} (4j - 4n - 2) \lambda_j (-w)^{j-i}$ becomes invertible in $U_{\lambda}$.

The restricted deformation family $X_{\lambda'}$ for $\lambda' \in \Sigma_i$ induces a Kodaira-Spencer map
$\kappa_{\lambda}: T_{\Sigma_i, \lambda} \to U_{\lambda}$.
Since $T_{\Sigma_i, \lambda}$ is a vector space which has a basis corresponding to monomials $y^2(-w)^j - (-w)^{j+2n+1}$ for $j \ge i$, 
we deduce that $\kappa_{\lambda} = 0$.
Since this is true for any $\lambda \in \Sigma_i$, we conclude that the restricted deformation family is locally trivial.
\end{proof}

\begin{Rem}
We have an alternative proof of the above theorem which uses Theorem~\ref{MY} below.
Indeed the above proof shows that the commutative deformation algebras of hypersurfaces 
for $\lambda \in \Sigma_i$ are independent of the coefficients $\lambda_j$ for $j > i$. 
\end{Rem}

\begin{Cor}
There are only $2n+1$ isomorphism classes in the deformation family $X_{\lambda}$.
They are isomorphic to $X_0$ defined by $F_0 = 0$ or $X_i$ defined by  
\[
F_i =  x^2 + y^3 + y^2(-w)^i + z^2w + yw^{2n+1} - (-w)^{i+2n+1} = 0
\]
for $1 \le i \le 2n$.
\end{Cor}

We calculate the NC deformation algebra of $\mathcal{O}_C$ on $Y_{\lambda}$:

\begin{Thm}\label{Laufer def}
The NC deformation algebra of $\mathcal{O}_C$ on $Y_{\lambda}$ is given by
\[
k\langle \langle a,b \rangle \rangle/(ab+ba,a^2+b^{2n+1}+\sum_{i=1}^{2n} \lambda_ib^{2i}).
\]
\end{Thm}

\begin{proof}
We add relations (\ref{Laufer substitution}) with the help of (\ref{u,v,w}) and (\ref{y}).
Then we have
\[
t=b^{2n}, \quad -a^2 = tb + \sum_{i=1}^{2n}\lambda_ib^{2i}, \quad ab+ba=0
\]
hence 
\[
a^2+b^{2n+1}+\sum_{i=1}^{2n} \lambda_ib^{2i}=0.
\]
We check that the relations of Theorem~\ref{univ def alg length 2} follow from these equations.
\[
\begin{split}
&ab^{2n+1}=-b^{2n+1}a, \\
&a(b^{2n+1}+\sum_{i=1}^{2n} \lambda_ib^{2i}) = (b^{2n+1}+\sum_{i=1}^{2n} \lambda_ib^{2i})a,
\end{split}
\]
hence 
\begin{equation}\label{ab2n+1}
0 = ab^{2n+1} = tab = tba. 
\end{equation}
\end{proof}

\begin{Rem}
(1) The normal bundle of $C$ is given by $N_{C/Y_{\lambda}} \cong \mathcal{O}_C(1) \oplus \mathcal{O}_C(-3)$.
Hence $\dim \text{Ext}^1(\mathcal{O}_C,\mathcal{O}_C) = \dim H^0(N_{C/Y_{\lambda}}) = 2$ and 
$\dim \text{Ext}^2(\mathcal{O}_C,\mathcal{O}_C) = \dim H^1(N_{C/Y_{\lambda}}) = 2$. 
Thus the cotangent space of the NC deformation algebra is generated by $a,b$ and there are $2$ relations.
The quadratic terms of the relations are anti-symmetric: $ab+ba, a^2$. 

(2) When $\lambda_i = 0$ for all $i$, then the equation is weighted homogeneous with weights
$wt(a,b) = (2n+1,2)$.
Thus the isomorphism type of the NC deformation algebra does not change under the substitution
\[
\lambda_i \mapsto \alpha^{2i - 2n + 1} \lambda_i
\]
for $\alpha \in k$.

(3) Since $\dim Y_{\lambda} = 3$ and $(K_{Y_{\lambda}},C) = 0$, there is a {\em superpotential} $W$ (\cite{VdBergh-CY}).
It is expressed by using non-commutative variables $a,b,c,d,w$ which are generators of 
$\text{End}(R \oplus M)$, where $a,b \in \text{Hom}(M,M)$,
$c \in \text{Hom}(M,R)$, $d \in \text{Hom}(R,M)$ and $w \in \text{Hom}(R,R)$. 
It seems to be given by 
\[
W = \frac 12 dcdc + b^2dc + a^2b + dwc - \frac{(-w)^{n+1}}{n + 1} + \frac{b^{2n+2}}{2n + 2}
+ \sum \lambda_{i=1}^{2n} \frac{b^{2i+1}}{2i+1}.
\]
By cyclically differentiating $W$, we then obtain the relations among the variables: 
\[
\begin{split}
&ab+ba = 0, \\
&a^2+bdc+dcb+b^{2n+1}+\sum \lambda_ib^{2i} = 0, \\
&c(b^2+dc) = (b^2+dc)d = 0, \\
&cd + (-w)^n = 0.
\end{split}
\]
By putting $c=d=w=0$, we obtain the relations of the NC deformation algebra.

(4) As already remarked in \cite{Brown-Wemyss}, non-commutative associative algebras 
may be non-isomorphic even if
their abelianizations are isomorphic.
It may even happen that one is finite dimensional and the other is infinite dimensional.
 
We know that $k[[a,b]]/(a^2+b^2+b^3) \cong k[[a,b]]/(a^2+b^2)$, because $1+b$ is invertible.
But $k\langle \langle a,b \rangle \rangle/(ab+ba,a^2+b^2+b^3)$ is finite dimensional ($9$-dimensional),
and $k\langle \langle a,b \rangle \rangle/(ab+ba,a^2+b^2)$ is infinite dimensional.
Indeed we have an injective homomorphism
\[
\begin{split}
&k[[a^2]] = k[[a^2,b^2 ]]/(a^2+b^2)
= k\langle \langle a^2,b^2 \rangle \rangle/(a^2b^2-b^2a^2,a^2+b^2) \\
&\to k\langle \langle a,b \rangle \rangle/(ab+ba,a^2+b^2).
\end{split}
\]
\end{Rem}

We consider the following conjecture of Donovan and Wemyss: 

\begin{Conj}[\cite{Donovan-Wemyss}~Conjecture 1.4]\label{dw-conj}
Let $f_i: Y_i \to X_i$ ($i=1,2$) be flopping contractions of smooth $3$-folds whose exceptional loci are 
irreducible smooth rational curves $C_i$, and let $A_i$ be the NC deformation algebras of 
$\mathcal{O}_{C_i}$ on $Y_i$.
Then the completions of $X_i$ at the singular points $f_i(C_i)$ are isomorphic if and only if $A_i$ are
isomorphic.
\end{Conj}

Conjecture~\ref{dw-conj} has some partial positive answers in \cite{Hua-Toda} and \cite{Hua}.
Conjecture~\ref{dw-conj} seems more reasonable than it appears because it can be regarded
as a non-commutative generalization of the the following theorem:

\begin{Thm}[\cite{MY}]\label{MY}
Let $(X,0),(X',0) \subset \mathbf{C}^{n+1}$ be germs of hypersurfaces with isolated singularities at the  
origin defined by equations $f,f'$.
Then they are isomorphic (i.e., biholomorphically equivalent) if and only if 
their (commutative) deformation algebras are isomorphic:
\[
\mathcal{O}_{\mathbf{C}^{n+1},0}/(f,\frac{\partial f}{\partial z_0}, \dots, \frac{\partial f}{\partial z_n})
\cong \mathcal{O}_{\mathbf{C}^{n+1},0}/(f',\frac{\partial f'}{\partial z_0}, \dots, \frac{\partial f'}{\partial z_n})
\]
as $\mathbf{C}$-algebras, where $z_0,\dots,z_n$ are local coordinates on $\mathbf{C}^{n+1}$
at the origin.
\end{Thm}

We define
\[
\begin{split}
&A_0 = k\langle \langle a,b \rangle \rangle/(ab+ba,a^2+b^{2n+1}) \\
&A_i = k\langle \langle a,b \rangle \rangle/(ab+ba,a^2+b^{2n+1}+b^{2i})
\end{split}
\]
for $1 \le i \le 2n$. 
We denote by $A_i^{ab} = A_i/([A_i,A_i])$ their abelianizations.

We confirm Conjecture~\ref{dw-conj} for deformations of Laufer's flops
in Proposition~\ref{non-isomorphism1} and Theorem~\ref{non-isomorphism2}.
This is a generalization of \cite{Brown-Wemyss}~Theorem 4.7.

\begin{Prop}\label{non-isomorphism1}
$A_0^{ab},A_1^{ab},\dots,A_n^{ab}$ are not isomorphic to each other, while 
$A_0^{ab}, A_{n+1}^{ab}, \dots,A_{2n}^{ab}$ are isomorphic. 
\end{Prop}

\begin{proof}
We have 
\[
A_i^{ab} = \begin{cases} 
k[[a,b]]/(ab,a^2+b^{2n+1}), \,\,\, &i=0, \\
k[[a,b]]/(ab,a^2+b^{2i}), \,\,\, &1 \le i \le n, \\
k[[a,b]]/(ab,a^2+b^{2n+1}), \,\,\, &n+1 \le i \le 2n.
\end{cases}
\]
Thus $\dim A_0^{ab} = 2+2n+1$ and $\dim A_i^{ab} = 2+2i$ for $1 \le i \le n$.
\end{proof}

\begin{Thm}\label{non-isomorphism2}
$A_0,A_{n+1},\dots,A_{2n}$ have the same dimension $6n+3$ as $k$-linear spaces, and
are non-isomorphic to each other as associative $k$-algebras.
\end{Thm} 

\begin{proof}
We consider $A_0$ and $A_{n+i}$ for $1 \le i \le n$.
We have $ab^2=b^2a$.

We claim that $a^3 = ab^{2n+1} = b^{4n+2} = 0$ in $A_{n+i}$. 
Indeed  
\[
a^3 + ab^{2n+2i} = - ab^{2n+1} = b^{2n+1}a = - a^3 - b^{2n+2i}a = -a^3 - ab^{2n+2i},
\]
hence $ab^{2n+1} = 0$.  
Then $a^3 = - ab^{2n+1} - ab^{2n+2i} = 0$. 
We have 
\[
b^{4n+2} = - a^2b^{2n+1} - b^{4n+2i+1} = - b^{4n+2i+1} = b^{4n+4i} = \dots.
\]
Since $\bigcap_{m=1}^{\infty} (b^m) = 0$, we obtain $b^{4n+2} = 0$.
We have $a^3 = ab^{2n+1} = b^{4n+2} = 0$ also in $A_0$.

Since $a^2 + b^{2n+1}$ is weighted homogeneous with $wt(a,b) = (2n+1,2)$, $A_0$ is a graded algebra.
$A_0$ has a $k$-linear basis $a^sb^t$ for $0 \le s \le 2$ and $0 \le t \le 2n$. 
Indeed since they have all different weights:
\[
\begin{split}
&wt(1,b,\dots,b^n,a,b^{n+1},ab,b^{n+2},ab^2,\dots,b^{2n}, \\
&ab^n, a^2, ab^{n+1}, a^2b, \dots, ab^{2n}, a^2b^n, a^2b^{n+1}, \dots, a^2b^{2n}) \\
&= (0,2,\dots,2n, 2n+1, 2n+2, 2n+3, 2n+4, 2n+5, \dots,4n,\\
&4n+1, 4n+2, 4n+3, 4n+4, \dots, 6n+1, 6n+2, 6n+4, \dots, 8n+2)
\end{split}
\]
they are linearly independent.

Assuming that there is an isomorphism $f: A_{n+i} \to A_0$ for some $i$, we will derive a contradiction.

Since $a^3 = 0$ in $A_{n+i}$, we have $f(a)^3 = 0$.
Since $wt(b) = 2$ is minimal, $(b + \dots)^3 \ne 0$, hence $b \not\in f(a)$, i.e., the monomial 
$b$ does not appear in $f(a)$.
Since $f(\mathfrak m) = \mathfrak m$ for maximal ideals $\mathfrak m$, we have $a \in f(a)$.

We claim that $b^k \not\in f(a)$ for $1 \le k \le 2n$.
We proceed by induction on $k$.
Assume that $b^k \not\in f(a)$ for $k < k_0$ for some $k_0 \le 2n$.
If $k_0 \le n$, then $wt(b^{k_0}) < wt(a)$, and we have
$b^{3k_0} \in (b^{k_0} + a + \dots)^3$, i.e., 
$b^{3k_0}$ does not cancel since its weight attains the minimum, 
where we note that $3k_0 \le 4n+1$.
Here we omitted coefficients as long as they are non-zero, because we are dealing with only monomials. 
Since $f(a)^3 = 0$, we deduce that $b^{k^0} \not\in f(a)$.
Similarly, if $k_0 > n$, then $a^2b^{k_0} \in (a + b^{k_0} + \dots)^3$, because $wt(a) < wt(b^{k_0})$.
Therefore $b^{k_0} \not\in f(a)$.

We have $f(a)f(b)+f(b)f(a)=0$.
We claim that $a \not\in f(b)$.
Indeed if $a \in f(b)$, then since $a \in f(a)$ and $b^k \not\in f(a)$ for $1 \le k \le 2n$, 
$a^2 \in f(a)f(b)+f(b)f(a)$ is not cancelled, a contradiction.
Thus $a \not\in f(b)$.
Since $f(\mathfrak m) = \mathfrak m$, we have $b \in f(b)$.

We claim that 
$b^{2k} \not\in f(b)$ for all $1 \le k \le n$.
Otherwise, $ab^{2k} \in f(a)f(b)+f(b)f(a)$ does not cancel for some $k$.

We use $f(a)^2 + f(b)^{2n+1} + f(b)^{2n+2i} = 0$, and we look at a monomial $a^2b^{2i-1} = - b^{2n+2i}$.
We have $a^2b^{2i-1} \in f(b)^{2n+2i}$, because $f(b) = b + \dots$.
On the other hand, we claim that $a^2b^{2i-1} \not\in f(a)^2$.
Indeed, since $b^k \not\in f(a)$ for $1 \le k \le 2n$, we need to look at $(a + ab^{2i-1})^2$.
But 
\[
(a + ab^{2i-1})^2 = a^2 + a^2b^{2i-1} - a^2b^{2i-1} - a^2b^{4i-2} =  a^2 - a^2b^{4i-2},
\] 
and $a^2b^{2i-1}$ disappears.

We claim that $b^{2n+2i} \not\in f(b)^{2n+1}$.
The form $(b + ab + \dots)^{2n+1}$, where $a$ does not appear inside the parentheses, contains $a^2b^m$ with $m \ge 2n+1$, and  $a^2b^{2i-1}$ does not appear. 
The form $(b + b^{m_1}+\dots+b^{m_s}+ \dots)^{2n+1}$, where the $m_k$ are odd, does not contain $b^{2n+2i}$, 
because an odd sum of odd numbers is odd and cannot be $2n+2i$.
Therefore $b^{2n+2i} \not\in f(b)^{2n+1}$.
But this is a contradiction with $f(a)^2 + f(b)^{2n+1} + f(b)^{2n+2i} = 0$.
Thus we conclude that there is no isomorphism $f: A_{n+i} \to A_0$.

\vskip 1pc

Next we generalize the above argument to prove that there is no isomorphism $f: A_{n+i} \to A_{n+j}$ for 
$1 \le i < j \le n$.
We claim first that the monomials $a^sb^t$ for $0 \le s \le 2$ and $0 \le t \le 2n$ are $k$-linear basis of $A_{n+j}$.
Indeed, by using the relation $ba = -ab$, any monomials are written as $a^sb^t$ for some integers $s,t$. 
Since $a^3 = ab^{2n+1} = b^{4n+2} = 0$, all possible non-zero monomials are 
$b^t$ for $0 \le t \le 4n+1$, $ab^t$ for $0 \le t \le 2n$ and $a^2b^t$ for $0 \le t \le 2n$.
The only relations among them are given by $a^2b^t + b^{2n+1+t} + b^{2n+2j+t} = 0$, where the last terms for 
$t \ge 2n-2j+2$ vanish.
Thus $b^{2n+1+t}$ for $0 \le t$ are expressed by other monomials while $b^t$ for $0 \le t \le 2n$ are not.
The multiplication of $A_{n+j}$ is also determined by this rule.

We have still $f(a)^3 = 0$.
Then we have again $b \not\in f(a)$, hence $a \in f(a)$.

We prove again that $b^k \not\in f(a)$ for $1 \le k \le 2n$ by induction on $k$.
Assume that $b^k \not\in f(a)$ for $k < k_0$ for some $k_0 \le 2n$.
If $k_0 \le n$, then $b^{3k_0} \in (b^{k_0} + a + \dots)^3$ as before, and obtain a contradiction if $3k_0 \le 2n$.
If $3k_0 > 2n$, then we rewrite
\[
\begin{split}
&b^{3k_0} = - a^2b^{3k_0 - 2n-1} - b^{3k_0+2j-1} \\
&= - a^2b^{3k_0 - 2n-1} + a^2b^{3k_0+2j-2n-2} + b^{3k_0 + 4j-2} = \dots
\end{split}
\]
and obtain a contradiction.
Hence $b^{k^0} \not\in f(a)$.
Similarly, if $k > n$, then we consider $a^2b^{k_0} \in (a + b^{k_0} + \dots)^3$.
Therefore $b^{k_0} \not\in f(a)$.

We use again $f(a)f(b)+f(b)f(a)=0$.
We claim that $a \not\in f(b)$.
Indeed, if $a \in f(b)$, then $a^2 \in f(a)f(b)+f(b)f(a)$, since we have $a \in f(a)$ and $b^k \not\in f(a)$ for $1 \le k \le 2n$, 
a contradiction.
Thus $a \not\in f(b)$, hence $b \in f(b)$.

We have also that 
$b^{2k} \not\in f(b)$ for $1 \le k \le n$, otherwise $ab^{2k} \in f(a)f(b)+f(b)f(a)$.

We use $f(a)^2 + f(b)^{2n+1} + f(b)^{2n+2i} = 0$ again.
We look at a monomial $b^{2n+2i}$. 
We have
\[
b^{2n+2i} =  - a^2b^{2i-1} - b^{2n+2i+2j-1} = - a^2b^{2i-1} + a^2b^{2i + 2j - 2} - \dots.
\]
Hence $a^2b^{2i-1} \in f(b)^{2n+2i}$.
On the other hand, 
$a^2b^{2i-1} \not\in f(a)^2$, because 
\[
(a + ab^{2i-1})^2 = a^2 + a^2b^{2i-1} - a^2b^{2i-1} - a^2b^{4i-2} =  a^2 - a^2b^{4i-2}.
\]

We claim that $a^2b^{2i-1} \not\in f(b)^{2n+1}$.
The form $(b + ab)^{2n+1}$ contains $a^2b^m$ with $m \ge 2n+1$, but does not contain $a^2b^{2i-1}$. 
The form $(b + b^{m_1}+\dots+b^{m_s})^{2n+1}$ with odd $m_k$ does not contain $b^{2n+2i}$, 
because an odd sum of odd numbers is odd and cannot be $2n+2i$.
We note also that $j > i$.
Therefore we obtain a contradiction with $f(a)^2 + f(b)^{2n+1} + f(b)^{2n+2i} = 0$ again, and 
we conclude that there is no isomorphism $f: A_{n+i} \to A_{n+j}$.
\end{proof}

%%%%%%%%%%%%%%%%%%%%%%%%%%%%%%%%%
%%%%%%%%%%%%%%%%%%%%%%%%%%%%%%%%%
%%%%%%%%%%%%%%%%%%%%%%%%%%%%%%%%%
\section{Example: universal flopping contraction of higher length}

By using Karmazyn \cite{Karmazyn}, we can describe NC deformation algebras of the reduced fiber $\mathcal{O}_C$ 
for the universal flopping contractions 
in the case of higher length $l \ge 3$.
We recall a description of the endomorphism algebras of the tilting bundles: 

\begin{Thm}[\cite{Karmazyn}~Theorem 1.3]\label{Karmazyn}
Let $\tilde f: \tilde Y \to \tilde X = Spec(R)$ be a universal flopping contraction of length $l$ for $l = 1,2,3,4,5,6$, and let 
$A = \text{End}(\mathcal{O}_{\tilde Y} \oplus M)$ be the endomorphism algebra of a tilting generator
$\mathcal{O}_{\tilde Y} \oplus M$ of $D^b(\text{coh}(\tilde Y))$ over $\tilde X$.
Then $A$ is a quiver algebra $A = H[Q]/I$ over a polynomial algebra $H$ 
with relations $I$ as follows, where there are two vertices $v_0,v_1$ of the quiver $Q$ with the 
corresponding idempotents $e_0,e_1$ and edges 
$a_0,a_1^*,a \in \text{Hom}(v_0,v_1)$, $a_0^*,a_1,a^* \in \text{Hom}(v_1,v_0)$ and $b,c,d \in \text{Hom}(v_1,v_1)$.

(1) $l=1$, $H = k[t]$, and $I$ is generated by 
\[
a^*_0a_0-a_1a^*_1 = te_0, \quad 
a^*_1a_1-a_0a^*_0 = -te_1.
\]

(2) $l=2$, $H = k[t,u_1,u_2,u_3]$, and $I$ is generated by 
\[
\begin{split}
&a^*a = te_0, \quad b^2 = u_1e_1, \quad c^2 = u_2e_1, \quad d^2 = u_3e_1, \\
&aa^* + b + c + d = \frac 12 te_1.
\end{split}
\]

(3) $l=3$, $H = k[t,u_1,u_2,u_3,u_4,u_5]$, and $I$ is generated by 
\[
\begin{split}
&a^*d = ta^*, \quad da = ta, \quad a^*a = (t^2-u_5)e_0, \quad aa^* = d^2 - u_5e_1, \\
&b^3 = u_2b + u_1e_1, \quad c^3 = u_4c + u_3e_1, \quad b + c + d = \frac 13 te_1.
\end{split}
\]

(4) $l=4$, $H = k[t,u_1,u_2,u_3,u_4,u_5,u_6]$, and $I$ is generated by 
\[
\begin{split}
&a^*d = ta^*, \quad da = ta, \quad a^*a = (t^3 - u_6t - u_5)e_0, \\
&aa^* = d^3 - u_6d - u_5e_1, \quad b^2 = u_1e_1, \quad c^4 = u_4c^2 + u_3c + u_2e_1, \\
&b+c+d = \frac 14 te_1.
\end{split}
\]

(5) $l=5$, $H = k[t,u_1,u_2,u_3,u_4,u_5,u_6,u_7]$, and $I$ is generated by 
\[
\begin{split}
&a^*d = ta^*, \quad da = ta, \quad a^*a = (t^4 - u_3t^2 - u_2t - u_1)e_0, \\ 
&aa^* = d^4 - u_3d^2 - u_2d - u_1e_1, \\
&cbc + bc^2 + b^3c = - u_7bc - u_6c - u_4e_1, \\ 
&(c+b^2)^2 + bcb = - u_7(c+b^2) - u_6b - u_5e_1, \quad d - b = \frac 15 te_1.
\end{split}
\]

(6) $l=6$, $H = k[t,u_1,u_2,u_3,u_4,u_5,u_6,u_7]$, and $I$ is generated by 
\[
\begin{split}
&a^*d = ta^*, \quad da = ta, \quad a^*a = (t^5 - u_7t^3 - u_6t^2 - u_5t - u_4)e_0, \\ 
&aa^* = d^5 - u_7d^3 - u_6d^2 - u_5d - u_4e_1, \quad b^2 = u_1e_1, \\ 
&c^3 = u_3c + u_2e_1, \quad b+c+d = \frac 16 te_1.
\end{split}
\]
\end{Thm}

We calculate the NC deformation algebra of the reduced central fiber by using the above theorem:

\begin{Thm}\label{higher length}
Let $\tilde f: \tilde Y \to \tilde X = Spec(R)$ be as in Theorem~\ref{Karmazyn}.
Then the NC deformation algebra $A_{\text{def}}$ of the reduced central fiber $\mathcal{O}_C$ of $\tilde f$
is the completion of a quotient algebra $A/Ae_0A$, and is given by the following:

(1) $l = 1$: $A_{\text{def}} = k$.

(2) $l = 2$: $A_{\text{def}} = k[[t]] \langle \langle b,c \rangle \rangle/(tbc-tcb, \,\,\, bc^2-c^2b, \,\,\, b^2c-cb^2)$.

(3) $l = 3$: 
\[
\begin{split}
&A_{\text{def}} = k[[t,u_2,u_4]] \langle \langle b,c \rangle \rangle/(u_2(bc-cb) - (b^3c-cb^3), \\
&u_4(bc-cb) - (bc^3-c^3b), \,\,\, 3(bc^2-c^2b) + 3(b^2c-cb^2) - 2(tbc - tcb)).
\end{split}
\]

(4) $l = 4$: 
\[
\begin{split}
&A_{\text{def}} = k[[t,u_3,u_4,u_6]] \langle \langle b,c \rangle \rangle/(b^2c-cb^2, \\ 
&u_3(bc-cb) + u_4(bc^2-c^2b) - (bc^4-c^4b), \\
&(16u_6- 3t^2)(bc-cb) + 12t(bc^2-c^2b) \\
&- 16(b^3c-cb^3 + b^2c^2 -c^2b^2+ bcbc-cbcb  -c^3b + bc^3)).
\end{split}
\]

(5) $l = 5$: 
\[
\begin{split}
&A_{\text{def}} = k[[t,u_2,u_3,u_6,u_7]] \langle \langle b,c \rangle \rangle/(u_7(bc-cb) + (bc^2-c^2b) + (b^3c - cb^3), \\
&u_6(bc-cb) + u_7(b^2c - bcb) + (bcbc - cbcb) + (b^2c^2 - bc^2b) + (b^4c - b^3cb), \\
&u_2(bc-cb) + u_3(b^2c-cb^2 + (2/5)t(bc-cb)) - (4/125)t^3(bc-cb) \\
&- (6/25)t^2(b^2c-cb^2) - (4/5)t(b^3c^cb^3) - (b^4c-cb^4)).
\end{split}
\]

(6) $l = 6$: 
\[
\begin{split}
&A_{\text{def}} = k[[t,u_3,u_5,u_6,u_7]] \langle \langle b,c \rangle \rangle/(b^2c-cb^2, \,\,\, u_3(bc-cb) - bc^3-c^3b, \\
&(- (5/6^4)t^4 + (1/12)t^2u_7 + (1/3)tu_6 + u_5)(bc - cb) \\
&+ ((10/6^3)t^3 - (1/2)t - u_6)(b^2c - cb^2 + bc^2 - c^2b) \\
&+ (- (10/6^2)t^2 + (1/6)u_7) \\
&(b^3c-cb^3 + b^2c^2 - c^2b^2 + bcbc-cbcb + bc^3 - c^3b) \\
&+ (5/6)t (c^4b-bc^4 + c^3b^2-b^2c^3 + c^2bcb - bcbc^2 + cbc^2b - bc^2bc \\
&+ c^2b^3 - b^3c^2 + cbcb^2-bcb^2c + cb^2cb - b^2cbc + cb^4 - b^4c) \\
&- (b^5c - cb^5) \\
&+ (b^4c^2-c^2b^4 +b^3cbc-cbcb^3 + bcb^3c - cb^3cb + b^2cb^2c - cb^2cb^2) \\
&+ (b^3c^3 - c^3b^3 + b^2cbc^2 - cbc^2b^2 + b^2c^2bc - c^2bcb^2 + bcb^2c^2 - cb^2c^2b \\
&+ bc^2b^2c - c^2b^2cb + bcbcbc - cbcbcb) \\
&+ (b^2c^4-c^4b^2 + bcbc^3 - cbc^3b + bc^2bc^2 - c^2bc^2b + bc^3bc - c^3bcb) \\
&+ (bc^5 - c^5b)).
\end{split}
\]
\end{Thm}

\begin{proof}
We note that the two-sided ideal $Ae_0A$ is generated by those endomorphisms of 
$\mathcal{O}_{\tilde Y} \oplus M$
which factor through $\mathcal{O}_{\tilde Y}$.
Therefore its completion is the NC deformation algebra by Theorem~\ref{partial}.

We deduce our result by changing variables from $A/Ae_0A$.
(1) is clear because $C$ is a $(-1,-1)$-curve, and is rigid inside $\tilde Y$.

\vskip 1pc

(2) This is already proved in Theorem~\ref{univ def alg length 2}, but we give an alternative proof.
We obtain
\[
A/Ae_0A = H \langle b,c,d \rangle/(b^2-u_1, \,\,\, c^2-u_2, \,\,\, d^2-u_3, \,\,\, b+c+d-t/2).
\]
We remove $d$ by using $d = t/2 - b - c$.
Since $b^2$, $c^2$ and $d^2 = t^2/4 - (tb + tc) + b^2+c^2 + (bc+cb)$ are in the center
of $A/Ae_0A$, so is $(bc+cb) - (tb+tc)$.
Then $(bbc + bcb - bcb - cbb) - t(bc-cb) = 0$.
Therefore we have $tbc - tcb = 0$.

We have $\dim \tilde Y = 6$ and $(K_{\tilde Y},C) = 0$.
Since $C \subset \tilde Y$ is a deformation of a $(1,-3)$-curve, there is an exact sequence
\[
0 \to \mathcal{O}_C^3 \to N_{C/\tilde Y} \to \mathcal{O}_C(1) \oplus \mathcal{O}_C(-3) \to 0
\]
for a normal bundle $N_{C/\tilde Y}$.
Since the pair $C \subset \tilde Y$ has no more deformation, 
we deduce that $H^1(T_{\tilde Y} \vert_C) = 0$ for the tangent bundle $T_{\tilde Y}$.
Indeed if $H^1(T_{\tilde Y} \vert_C) \ne 0$, then $C \subset \tilde Y$ has a non-trivial deformation.
Thus there is no factor of degree $\le -2$ in $T_{\tilde Y} \vert_C$, and 
we obtain 
\[
N_{C/\tilde Y} \cong \mathcal{O}_C(1) \oplus \mathcal{O}_C \oplus \mathcal{O}_C(-1)^3.
\] 
Then 
\[
\bigwedge^2 N_{C/\tilde Y} \cong \mathcal{O}_C(1) \oplus \mathcal{O}_C^3 \oplus \mathcal{O}_C(-1)^3 \oplus \mathcal{O}_C(-2)^3.
\]
Therefore 
\[
\dim H^0(N_{C/\tilde Y}) = 3, \quad \dim H^1(N_{C/\tilde Y})=0, \quad \dim H^0(\bigwedge^2 N_{C/\tilde Y}) = 5.
\]
We have $3$ generators $t,b,c$ of the maximal ideal of $A_{\text{def}}$, and
they satisfy $5$ relations.
Since 
\[
\dim (\text{Im}(\bigwedge^2 H^0(N_{C/\tilde Y}) \to H^0(\bigwedge^2 N_{C/\tilde Y}))) = 2,
\]
there are only $2$ relations of order $2$. 

Since $t,a^2,b^2$ are in the center, we have relations 
$at-ta=bt-tb=tab-tba=ab^2-b^2a=a^2b-ba^2=0$.
We have thus $3$ linearly independent relations of order $3$ besides $2$ relations of order $2$.
It follows that there are no more relations.
Therefore we have
\[
\begin{split}
&A_{\text{def}} = k\langle \langle t,b,c \rangle \rangle/(at-ta, \,\,\, bt-tb, \,\,\, tab-tba, \,\,\, ab^2-b^2a, \,\,\, a^2b-ba^2) \\
&= k[[t]] \langle \langle b,c \rangle \rangle/(tab-tba, \,\,\, ab^2-b^2a, \,\,\, a^2b-ba^2).
\end{split}
\]

\vskip 1pc

(3) We obtain
\[
A/Ae_0A = H \langle b,c,d \rangle/(b^3-u_2b-u_1, \,\,\, c^3-u_4c-u_3, \,\,\, d^2-u_5, \,\,\, b+c+d-t/3).
\]
We remove $d$ by using $d = t/3 - b - c$.

Since $b^3-u_2b$ and $c^3-u_4c$ are in the center, 
we have 
\[
\begin{split}
&0 = (b^3-u_2b)c - c(b^3-u_2b) = - u_2(bc-cb) + (b^3c-cb^3), \\
&0 = (c^3-u_4c)b - b(c^3-u_4c) = u_4(bc-cb) - (bc^3-c^3b).
\end{split}
\]
Since $d^2 = t^2/9 - \frac 23(tb + tc) + b^2+c^2 + (bc+cb)$ is in the center, so is 
$b^2 + c^2 + (bc+cb) - \frac 23 (tb+tc)$.
Hence 
\[
(bc^2-c^2b) + (bbc + bcb - bcb - cbb) - \frac 23 t(bc-cb) = 0.
\]
Thus $3(bc^2-c^2b) + 3(b^2c-cb^2) - 2(tbc - tcb) = 0$.

We have $N_{C/\tilde Y} \cong \mathcal{O}_C(1) \oplus \mathcal{O}_C^3 \oplus \mathcal{O}_C(-1)^3$ as before.
Since $\bigwedge^2 N_{C/\tilde Y} \cong \mathcal{O}_C(1)^3 \oplus \mathcal{O}_C^{3+3} \oplus \mathcal{O}_C(-1)^9 \oplus 
\mathcal{O}_C(-2)^3$, 
we have
\[
\dim H^0(N_{C/\tilde Y}) = 5, \quad \dim H^1(N_{C/\tilde Y})=0, \quad \dim H^0(\bigwedge^2 N_{C/\tilde Y}) = 12.
\]
Thus we have $5$ generators $t,u_2,u_4,b,c$ and $12$ relations.
Since 
\[
\dim (\text{Im}(\bigwedge^2 H^0(N_{C/\tilde Y}) \to H^0(\bigwedge^2 N_{C/\tilde Y}))) = 9,
\]
there are $9$ relations of order $2$, which are 
commutations with central variables.
We already have $3$ relations of order $3$, and there are no more relations.

\vskip 1pc

(4) We obtain
\[
A/Ae_0A = H \langle b,c,d \rangle/(b^2-u_1, \,\,\, c^4-u_4c^2-u_3c-u_2, \,\,\, d^3-u_6d-u_5, \,\,\, b+c+d-t/4).
\]
We remove $d$ by using $d = t/4 - b - c$.

We have $N_{C/\tilde Y} \cong \mathcal{O}_C(1) \oplus \mathcal{O}_C^4 \oplus \mathcal{O}_C(-1)^3$, 
$\bigwedge^2 N_{C/\tilde Y} \cong \mathcal{O}_C(1)^4 \oplus \mathcal{O}_C^{6+3} \oplus \mathcal{O}_C(-1)^{12} \oplus \mathcal{O}_C(-2)^3$.
Hence 
\[
\dim H^0(N_{C/\tilde Y}) = 6, \quad \dim H^1(N_{C/\tilde Y})=0, \quad \dim H^0(\bigwedge^2 N_{C/\tilde Y}) = 17.
\]
Thus we have $6$ generators $t,u_3,u_4,u_6,b,c$ and $17$ relations.
Since 
\[
\dim (\text{Im}(\bigwedge^2 H^0(N_{C/\tilde Y}) \to H^0(\bigwedge^2 N_{C/\tilde Y}))) = 14,
\]
there are $14$ relations of order $2$, which are 
commutations with central variables.
We need $3$ more relations of order $\ge 3$.

Since $b^2$ and $c^4-u_4c^2-u_3c$ are in the center, we have 
\[
\begin{split}
&0 = b^2c-cb^2, \\
&0 = b(c^4-u_4c^2-u_3c) - (c^4-u_4c^2-u_3c)b \\
&= -u_3(bc-cb) - u_4(bc^2-c^2b)+(bc^4-c^4b).
\end{split}
\]
Since 
\[
d^3-u_6d = (t/4)^3 - (3/16)t^2(b+c) + (3/4)t(b+c)^2 - (b+c)^3 - u_6(t/4 - b - c)
\]
is in the center, so is  
\[
(3/4)t(c^2+bc+cb) - (b^3+c^3 + bbc + bcb + cbb + cbc + ccb + bcc) + (u_6 - (3/16)t^2)(b + c).
\]
Thus 
\[
\begin{split}
&0 = 16((3/4)t(bc^2-c^2b) - (b^3c-cb^3 + b^2c^2 -c^2b^2+ bcbc-cbcb  -c^3b + bc^3) \\
&+ (u_6- (3/16)t^2)(bc-cb)) \\
&= (16u_6- 3t^2)(bc-cb) + 12t(bc^2-c^2b) \\
&- 16(b^3c-cb^3 + b^2c^2 -c^2b^2+ bcbc-cbcb  -c^3b + bc^3).
\end{split}
\]
Therefore we have our claim.

\vskip 1pc

(5) We obtain
\[
\begin{split}
&A/Ae_0A = H \langle b,c,d \rangle/(cbc + bc^2 + b^3c + u_7bc + u_6c + u_4, \\
&(c+b^2)^2 + bcb + u_7(c+b^2) + u_6b + u_5, \,\,\, d^4-u_3d^2-u_2d-u_1, \,\,\, - b+d-t/5).
\end{split}
\]
We remove $d$ by using $d = t/5 + b$.

We have $N_{C/\tilde Y} \cong \mathcal{O}_C(1) \oplus \mathcal{O}_C^5 \oplus \mathcal{O}_C(-1)^3$, 
$\bigwedge^2 N_{C/\tilde Y} \cong \mathcal{O}_C(1)^5 \oplus \mathcal{O}_C^{10+3} \oplus \mathcal{O}_C(-1)^{15} 
\oplus \mathcal{O}_C(-2)^3$.
Hence 
\[
\dim H^0(N_{C/\tilde Y}) = 7, \quad \dim H^1(N_{C/\tilde Y})=0, \quad \dim H^0(\bigwedge^2 N_{C/\tilde Y}) = 23.
\]
Thus we have $7$ generators $t,u_2,u_3,u_6,u_7,b,c$ and $23$ relations.
Since 
\[
\dim (\text{Im}(\bigwedge^2 H^0(N_{C/\tilde Y}) \to H^0(\bigwedge^2 N_{C/\tilde Y}))) = 20,
\]
there are $20$ relations of order $2$, which are 
commutations with central variables.
We need $3$ more relations of order $\ge 3$.

We have 
\[
\begin{split}
&0 = b(cbc + bc^2 + b^3c + u_7bc + u_6c) - (cbc + bc^2 + b^3c + u_7bc + u_6c)b \\
&= u_6(bc-cb) + u_7(b^2c-bcb) + (bcbc - cbcb) + (b^2c^2 - bc^2b) + (b^4c - b^3cb), \\
&0 = b((c+b^2)^2 + bcb + u_7(c+b^2) + u_6b) - ((c+b^2)^2 + bcb + u_7(c+b^2) + u_6b)b \\
&= (bc^2-c^2b) + (bcb^2+b^3c - cb^3 - b^2cb) + (b^2cb - bcb^2) + u_7(bc-cb) \\
&= u_7(bc-cb) + (bc^2-c^2b) + (b^3c - cb^3).
\end{split}
\]
We have $d^4-u_3d^2-u_2d
= (t/5 + b)^4 - u_3(t/5 + b)^2 - u_2(t/5 + b)$.
Hence $b^4 + (4/5)tb^3 + (6/25)t^2b^2 + (4/125)t^3b - u_3(b^2 + (2/5)tb) - u_2b$ is in the center, 
and we obtain
\[
\begin{split}
&0 = c(b^4 + (4/5)tb^3 + (6/25)t^2b^2 + (4/125)t^3b - u_3(b^2 + (2/5)tb) - u_2b) \\
&- (b^4 + (4/5)tb^3 + (6/25)t^2b^2 + (4/125)t^3b - u_3(b^2 + (2/5)tb) - u_2b)c \\
&= u_2(bc-cb) + u_3(b^2c-cb^2 + (2/5)t(bc-cb)) - (4/125)t^3(bc-cb) - (6/25)t^2(b^2c-cb^2) \\
&- (4/5)t(b^3c^cb^3) - (b^4c-cb^4)
\end{split}
\]
and our result.

\vskip 1pc

(6) We obtain
\[
\begin{split}
&A/Ae_0A = H \langle b,c,d \rangle/b^2-u_1, \,\,\, c^3-u_3c-u_2, \\
&d^5-u_7d^3-u_6d^2 - u_5d - u_4, \,\,\, b+c+d-t/6).
\end{split}
\]
We remove $d$ by using $d = t/6 - b -c$.

We have $N_{C/\tilde Y} \cong \mathcal{O}_C(1) \oplus \mathcal{O}_C^5 \oplus \mathcal{O}_C(-1)^3$, 
$\bigwedge^2 N_{C/\tilde Y} \cong \mathcal{O}_C(1)^5 \oplus \mathcal{O}_C^{10+3} \oplus \mathcal{O}_C(-1)^{15} \oplus \mathcal{O}_C(-2)^3$.
Hence 
\[
\dim H^0(N_{C/\tilde Y}) = 7, \quad \dim H^1(N_{C/\tilde Y})=0, \quad \dim H^0(\bigwedge^2 N_{C/\tilde Y}) = 23.
\]
Thus we have $7$ generators $t,u_3,u_5,u_6,u_7,b,c$ and $23$ relations.
Since 
\[
\dim (\text{Im}(\bigwedge^2 H^0(N_{C/\tilde Y}) \to H^0(\bigwedge^2 N_{C/\tilde Y}))) = 20,
\]
there are $20$ relations of order $2$, which are 
commutations with central variables.
We need $3$ more relations of order $\ge 3$.

We have 
\[
0 = b(c^3-u_3c) - (c^3-u_3c)b = - u_3(bc-cb) + bc^3-c^3b.
\]
Since 
\[
d^5-u_7d^3-u_6d^2 - u_5d = (t/6-b-c)^5-u_7(t/6-b-c)^3-u_6(t/6-b-c)^2 - u_5(t/6-b-c)
\]
we calculate
\[
\begin{split}
&b(- (5/6^4)t^4(b+c) + (10/6^3)t^3(b+c)^2 - (10/6^2)t^2(b+c)^3 \\
&+ (5/6)t(b+c)^4 - (b+c)^5 + u_7((1/12)t^2(b+c) - (1/2)t(b+c)^2 \\
&+ (1/6)(b+c)^3) + u_6((1/3)t(b+c) - (b+c)^2) + u_5(b+c)) \\
&- (- (5/6^4)t^4(b+c) + (10/6^3)t^3(b+c)^2 - (10/6^2)t^2(b+c)^3 \\
&+ (5/6)t(b+c)^4 - (b+c)^5 + u_7((1/12)t^2(b+c) - (1/2)t(b+c)^2 \\
&+ (1/6)(b+c)^3) + u_6((1/3)t(b+c) - (b+c)^2) + u_5(b+c))b \\
&= (- (5/6^4)t^4 + (1/12)t^2u_7 + (1/3)tu_6 + u_5)(bc - cb) \\
&+ ((10/6^3)t^3 - (1/2)t - u_6)(b^2c - bcb + bc^2 - c^2b) \\
&+ (- (10/6^2)t^2 + (1/6)u_7)(b^3c-cb^3 + b^2c^2 - c^2b^2 + bcbc-cbcb + bc^3 - c^3b) \\
&+ (5/6)t (c^4b-bc^4 + c^3b^2-b^2c^3 + c^2bcb - bcbc^2 \\
&+ cbc^2b - bc^2bc + c^2b^3 - b^3c^2 + cbcb^2-bcb^2c + cb^2cb - b^2cbc + cb^4 - b^4c) \\
&- (b^5c - cb^5) + (b^4c^2-c^2b^4 +b^3cbc-cbcb^3 + bcb^3c - cb^3cb + b^2cb^2c - cb^2cb^2) \\
&+ (b^3c^3 - c^3b^3 + b^2cbc^2 - cbc^2b^2 + b^2c^2bc - c^2bcb^2 + bcb^2c^2 - cb^2c^2b \\
&+ bc^2b^2c - c^2b^2cb + bcbcbc - cbcbcb) \\
&+ (bc^4-c^4b + bcbc^3 - cbc^3b + bc^2bc^2 - c^2bc^2b + bc^3bc - c^3bcb) + (bc^5 - c^5b)
\end{split}
\]
hence the result.
\end{proof}

\begin{Rem}
(1) In the case $l = 2$, the generators $b,c$ here are replaced by $a,b$ in \cite{Curto-Morrison}.

(2) The rank $r$ of the locally free sheaf $M$ is equal to $l$.
Indeed if $M$ is defined by an exact sequence $0 \to O^{r-1} \to M \to L \to 0$ with minimal $r$
from an invertible sheaf $L$ with $(L,C) = 1$ such that $R^1f_*M^* = 0$ as in \cite{VdBergh}, then we have $l = r$ as follows.
We define fattened structure sheaves $\mathcal{O}_{C_i}$ inductively by
$0 \to \mathcal{O}_C(-1) \to \mathcal{O}_{C_i} \to \mathcal{O}_{C_{i-1}} \to 0$ with $C_1 = C$.
Since $(L,C) = 1$, we have $\dim \text{Hom}(L,\mathcal{O}_C) = \dim \text{Ext}^1(L,\mathcal{O}_C) = 0$, 
hence $\dim \text{Hom}(M,\mathcal{O}_C) = r-1$.
Since $\dim \text{Hom}(M,\mathcal{O}_C(-1)) = 0$ and $\dim \text{Ext}^1(M,\mathcal{O}_C(-1)) = 1$, we have
$\dim \text{Hom}(M,\mathcal{O}_{C_i}) = \dim \text{Hom}(M,\mathcal{O}_{i-1}) - 1$.
Since $\dim \text{Hom}(M,\mathcal{O}_{C_l}) = 0$, we have $l = r$.

(3) We would like to ask whether the global NC deformation algebra of $\mathcal{O}_C$ is given in the form 
$H' \langle b,c \rangle/J$ over a polynomial algebra $H'$ with variables $t$ and some of the $u_i$, 
so that our (formal) deformation algebra is its completion.

(4) The abelianizations $A_{\text{def}}^{ab}$ are (commutative) formal power series rings for all $l$, because
$H^1(N_{C/\tilde Y}) = 0$.
That is why the generators of the defining ideals of the NC deformation rings are commutators. 
\end{Rem}

%%%%%%%%%%%%%%%%%%%%%%%%%%%%%%%%%%%%%%%%%%
%%%%%%%%%%%%%%%%%%%%%%%%%%%%%%%%%%%%%%%%%%
%%%%%%%%%%%%%%%%%%%%%%%%%%%%%%%%%%%%%%%%%%

Graduate School of Mathematical Sciences, University of Tokyo,
Komaba, Meguro, Tokyo, 153-8914, Japan. 

Morningside Center of Mathematics, 
Chinese Academy of Sciences, 
Haidian District, Beijing, China 100190

Department of Mathematical Sciences, 
Korea Advanced Institute of Science and Technology, 
291 Daehak-ro, Yuseong-gu, Daejeon 34141, Korea.

National Center for Theoretical Sciences, 
Mathematics Division, 
National Taiwan University, Taipei, 106, Taiwan.

kawamata@ms.u-tokyo.ac.jp


\begin{thebibliography}{}

\bibitem{AM}
Aspinwall, Paul S.; Morrison, David R. 
{\em Quivers from matrix factorizations}.
Commun. Math. Phys. 313 (2012), 607--633.
DOI 10.1007/s00220-012-1520-1.

\bibitem{Bondal}
Bondal, Alexey.
{\em Representations of associative algebras and coherent sheaves}. 
(Russian) Izv. Akad. Nauk SSSR Ser. Mat. 53 (1989), no. 1, 25--44; 
translation in Math. USSR-Izv. 34 (1989), no. 1, 23--42.

\bibitem{Brown-Wemyss}
Brown, Gavin; Wemyss, Michael. 
{\em Gopakumar-Vafa Invariants Do Not Determine Flops}.
Commun. Math. Phys. 361, 143--154 (2018).

\bibitem{Bridgeland}
Bridgeland, Tom.
{\em Flops and derived categories}. 
Invent. Math. 147 (2002), no. 3, 613--632.

\bibitem{Curto-Morrison}
Curto, C.,Morrison, D.R. 
{\em Threefold flops via matrix factorization}.
J. Alg. Geom., 22-4(2013), 599-627. 
doi.org/10.1090/S1056-3911-2013-00633-5.

\bibitem{Donovan-Wemyss}
Donovan, Will; Wemyss, Micheal.
{\em Noncommutative deformations and flops}.
Duke Math. J. Volume 165, Number 8 (2016), 1397--1474.

\bibitem{Eisenbud}
Eisenbud, D.
{\em Homological algebra on a complete intersection, with an application to group representations}.
Trans. Amer. Math. Soc. 260 (1980), 35--64.

\bibitem{Hua}
Hua, Zheng.
{\em Contraction algebra and singularity of three-dimensional flopping contraction}.
arXiv:1610.05467.

\bibitem{Hua-Toda}
Hua, Zheng; Toda, Yukinobu.
{\em Contraction algebra and invariants of singularities}.
Int. Math. Res. Notices 2018(10).
DOI: 10.1093/imrn/rnw333.

\bibitem{Karmazyn}
Karmazyn, Joseph.
{\em The length classification of threefold flops via noncommutative algebras}.
Advances Math. 343.
DOI: 10.1016/j.aim.2018.11.023.

\bibitem{Katz}
S. Katz.
{\em Genus zero Gopakumar-Vafa invariants of contractible curves}.
J. Differential. Geom. 79 (2008), 185--195.

\bibitem{KatzM}
S. Katz and D. R. Morrison.
{\em Gorenstein threefold singularities with small resolutions
via invariant theory of Weyl groups}. 
J. Alg. Geom. 1 (1992), 449--530.

\bibitem{hyperplane}
Kawamata, Yujiro.
{\em General hyperplane sections of nonsingular flops in dimension 3}.
Math. Res. Let. {\bf 1} (1994), 49--52.

\bibitem{NC multi}
Kawamata, Yujiro.
{\em On multi-pointed non-commutative deformations and Calabi-Yau threefolds}.
Compositio Math. 154 (2018), 1815--1842. 
doi:10.1112/S0010437X18007248.

\bibitem{NC perv}
Kawamata, Yujiro.
{\em Non-commutative deformations of simple objects in a category of perverse coherent sheaves}.
Selecta Math. 26, Article number: 43 (2020)
DOI: 10.1007/s00029-020-00570-w

\bibitem{NC formal}
Kawamata, Yujiro.
{\em On non-commutative formal deformations of coherent sheaves on an algebraic variety}.
to appear in EMS Surveys in Mathematical Sciences.

\bibitem{Laufer}
Laufer, H.B.
{\em On $\mathbf{CP}^1$ as an exceptional set}. 
In: Recent Developments in Several Complex Variables.
Fornaess, J.E. ed., Ann. of Math. Stud., Vol. 100. Princeton, NJ: Princeton University Press, 1981, 261--275.

\bibitem{MY}
Mather, John N.; Yau, Stephen S.-T.
{\em Classification of isolated hypersurface singularities by their moduli algebras}.
Invent. math. 69 (1982), 243--251.

\bibitem{P}
Pinkham, H. 
{\em Factorization of birational maps in dimension 3}.
Singularities (P. Orlik, ed.), Proc. Symp. Pure Math., vol. 40, part 2, American Mathematical Society, 
1983, 343--371.

\bibitem{Reid}
Reid, M.
{\em Minimal models of canonical 3-folds}. 
In: Algebraic Varieties and Analytic Varieties. Iitaka, S. ed., 
Adv. Stud. Pure Math., vol. 1. Tokyo: Kinokuniya, 1983, 131--180.

\bibitem{Rickard}
Rickard, Jeremy. 
{\em Morita theory for derived categories}. 
J. London Math. Soc. (2) 39 (1989), no. 3, 436--456.

\bibitem{Toda}
Toda, Yukinobu.
{\em Non-commutative width and Gopakumar-Vafa invariants}. 
Manuscripta Math. 148(2015), 521--533.

\bibitem{VdBergh}
Van den Bergh, Michel.
{\em Three-dimensional flops and noncommutative rings}.
Duke Math. J. 122 (2004), no. 3, 423--455.

\bibitem{VdBergh-CY}
Van den Bergh, Michel.
{\em Calabi-Yau algebras and superpotentials}.
Selecta Math. 21(2015), 555--603.

\bibitem{vGarderen}
Okke van Garderen.
{\em Donaldson-Thomas invariants of length 2 flops}.
arXiv:2008.02591.


\end{thebibliography}
\end{document}